\documentclass {amsart}
\pdfoutput=1
\usepackage{amsmath}
\usepackage{amssymb}
\usepackage{amsthm}
\usepackage{graphicx}
\usepackage{mathrsfs}
\usepackage{enumerate}
\usepackage[all]{xy}

\newcommand{\g}[1]{\mathfrak{#1}}
\newcommand{\ms}[1]{\mathscr{#1}}
\newcommand{\wt}[1]{\widetilde{#1}}
\newcommand{\wh}[1]{\widehat{#1}}
\newcommand{\op}[1]{\operatorname{#1}}

\theoremstyle{plain}
\newtheorem{theorem}{Theorem}[section]
\newtheorem*{main_theorem}{Main Theorem}
\newtheorem{lemma}[theorem]{Lemma}

\newtheorem{corollary}[theorem]{Corollary}
\newtheorem{proposition}[theorem]{Proposition}

\theoremstyle{definition}

\newtheorem{notation}[theorem]{Notation}
\newtheorem{example}[theorem]{Example}

\newtheorem{remark}[theorem]{Remark}

\begin{document}

\title{Homogeneous Interpolation and Some Continued Fractions}
\author{Ivan Petrakiev}
\address{Knight Capital Group, Jersey City, NJ}
\email{ivan.petrakiev@gmail.com}
\date{Nov 27, 2012}
\begin {abstract} We prove: if $d/m < 2280/721$, there is no curve of degree $d$ passing through $n=10$ general points with multiplicity $m$ in ${\bf P}^2$. Similar results are given for other special values of $n$. Our bounds can be naturally written as certain palindromic continued fractions.
\end{abstract}

\maketitle

%\begin{flushright}
%{\it ``Some men interpret nine memos.''}
%\end{flushright}

\section{Introduction}\label{sec_intro}

Denote by $\ms{L}(d,n,m)$ the linear system of degree $d$ curves in ${\bf P}^2$ passing through $n$ general points $\ms{P}_1,\dots,\ms{P}_n$ with multiplicity at least $m$. For $n \geq 9$, Nagata's conjecture (\cite{N}) predicts that $\ms{L}(d,n,m)$ is empty if $d < \sqrt{n} m$. The statement is clear if $n=k^2$ is a square, but remains widely open otherwise. A refined conjecture due to Harbourne-Hirschowitz further predicts that in fact $\ms{L}(d,n,m)$ is {\it non-special}, i.e. of expected dimension $\op{max}\{-1,v\}$, where $$v = d(d+3)/2-nm(m+1)/2$$ is the {\it virtual dimension} of $\ms{L}(d,n,m)$. We refer to (\cite{CD},\cite{CM1},\cite{CM2},\cite{E}) for background on the problem and some recent results.

In this paper we prove:

\begin{main_theorem}\label{thmB} Let $n$ be a non-square positive integer. Write $n=k^2+\alpha$ with $k = \lfloor \sqrt{n} \rfloor$. Assume that either:
\begin{enumerate}[i)] 
\item $n = 8, 10,12$, or
\item $k \geq 3$, $\alpha$ is even, $\alpha \mid 2n$.
\end{enumerate}
If the linear system $\ms{L}(d,n,m)$ is nonempty, then $d/m \geq c^{(2)}_n$, where
%$$c^{(2)}_n = k + \dfrac {\alpha} {2k + \dfrac {\alpha} { 2k + \dfrac { \alpha } { 2k + \dfrac \alpha k} } }.$$
$$c^{(2)}_n = k + \dfrac 1 {(2k/\alpha) + \dfrac {1\ \ } { 2k + \dfrac {\ \ 1} { (2k/\alpha) + \dfrac 1 k} } }.$$
\end{main_theorem}
Note that $c^{(2)}_n$ is a palindromic continued fraction with rational coefficients. The value $c^{(2)}_8 = \frac{48}{17}$ is well-known to be sharp (\cite{N}). The next few cases are $c^{(2)}_{10} = \frac{2280}{721}$, $c^{(2)}_{11}=\frac{660}{199}$, $c^{(2)}_{12}=\frac{336}{97}$, $c^{(2)}_{15}=\frac{120}{31}$ and $c^{(2)}_{18} = \frac{2448}{577}$. Since $\sqrt{10} - c^{(2)}_{10} \approx 3 \cdot 10^{-6}$, our bound for ten points is stronger than the bound $\frac{117}{37}$ obtained by Eckl (\cite{E}) and Ciliberto {\it et. al.} (\cite{CD}).

In the case $n=10,11,12$ we have a more refined result (Prop. \ref{prop_refinement}). As a striking application, we are able show that the linear system $\ms{L}(1499,10,474)$ is non-special (with $v=-1$).

The proof of the Theorem consists of two degenerations. First, we specialize the $n$ general points $\ms{P}_i$ in ${\bf P}^2$ to general points $P_i$ on a fixed curve $C$ of degree $k$. The problem is naturally reduced to an interpolation problem on a ruled surface $S= {\bf P}(\ms{E})$ where $\ms{E}$ is a semistable rank 2 vector bundle of degree $\alpha$ on $C$. Second, we specialize the $n$ points in $S$ to a curve $\Gamma$ of self-intersection $0$ (in general, this step requires a deformation of the underlying surface $S$). The methods in this paper extend our previous attempt in \cite{PE}. We were greatly influenced by the work of Ciliberto-Miranda in \cite{CM2}.

The paper is organized as follows. In Section \ref{sec_basic} we introduce a Basic Lemma that will be useful throughout the paper. In Section \ref{sec_ruled} we give background on ruled surfaces and elementary transforms. In Section \ref{sec_first} we perform the first degeneration and obtain a certain weak bound on $d/m$ for any $n \geq 9$. The construction is formalized in the next two sections. In Section \ref{sec_main} we sketch the proof of the Main Theorem. In each subsequent section we verify the theorem for specific values of $n$. In Section \ref{sec_refinement} we prove a certain refinement of the Main Theorem. In Appendix \ref{appendix_A} we review the indecomposable elliptic ruled surface of degree 1. Appendix \ref{appendix_B} has some auxiliary results on continued fractions.

\subsection*{Notation and Conventions} We work over $\mathbb C$. Following EGA IV.4, for given a subscheme $Y \subset X$ we denote by $\ms{N}_{Y/X} \cong \ms{I}_Y /\ms{I}_Y^2$ the {\it conormal} sheaf of $Y$. For {\it any} coherent sheaf $\ms{F}$ on $X$, we denote ${\bf P}(\ms{F}) = \op{Proj}(\oplus_{\mu \geq 0} \op{Sym}^\mu \ms{F})$.

\subsection*{Acknowledgments} The author is grateful to E. Cotterill for valuable comments and suggestions for improving the paper.

\section{Basic Lemma}\label{sec_basic}

The following elementary lemma is the key ingredient to several arguments in this paper.

\begin{lemma}[Basic Lemma]\label{basic_lemma} Let $C$ be a nonsingular curve embedded in a nonsingular projective variety $X$. Let $D$ be an effective divisor on $X$. Denote $\mu = \op{mult}_{C/X}(D)$, the multiplicity of vanishing of $D$ along $C$. Then, there is a natural injective morphism of sheaves $\ms{O}_{C}(-D) \hookrightarrow \op{Sym}^\mu \ms{N}_{C/X}$, where $\ms{N}_{C/X}$ is the conormal bundle of $C$ in $X$. 
\end{lemma}

\begin{proof} Let $\pi : X' \rightarrow X$ be the blowup of $C$. Then $\pi^*(D) = D' + \mu S$, where $S = {\bf P}(\ms{N}_{C/X})$ is the exceptional divisor of $\pi$ and $D'$ is the strict transform of $D$. We have $\ms{O}_S(S) = \ms{O}_S(-1)$, the tautological line bundle on $S$. So, $\pi_*(\ms{O}_S(\mu - D')) = \pi_* \pi^* \ms{O}_C(-D) = \ms{O}_C(-D)$ is naturally a subsheaf of $\pi_*\ms{O}_S(\mu) = \op{Sym}^\mu \ms{N}_{C/X}$ (also, it is a {\it subbundle} iff $D'|_S$ has no vertical components).
\end{proof}

We can use the lemma to give a lower bound on $\op{mult}_{C/X}(D)$ based on the ``local data'' $\ms{O}_C(-D)$. This takes a particularly simple form if $\ms{N}_{C/X}$ is a semistable vector bundle. Recall that, by definition, a vector bundle $\ms{E}$ on a curve $C$ is {\it semistable} if and only if for any subsheaf $\ms{F} \hookrightarrow \ms{E}$, we have $\op{slope}(\ms{F}) \leq \op{slope}(\ms{E})$. Here we denote $\op{slope}(\ms{E}) = \op{deg}(\ms{E}) / \op{rank}(\ms{E})$. For background on semistability, see e.g. \cite{PO}.

\begin{corollary}\label{basic_cor} In the above setting, suppose that $\ms{N}_{C/X}$ is semistable. Then
$$\op{slope}(\ms{N}_{C/X}) \op{mult}_{C/X}(D) \geq -C \cdot D.$$
\end{corollary}
\begin{proof} This follows from Basic Lemma together with the fact that, for any $\mu \geq 0$, $\op{Sym}^\mu \ms{N}_{C/X}$ is semistable of slope $\mu \cdot \op{slope}(\ms{N}_{C/X})$ (\cite{PO}, Thm. 10.2.1).
\end{proof}

\bigskip

\section{Ruled Surfaces and Semistability}\label{sec_ruled}

In this section we give some background on ruled surfaces and elementary transforms. Our reference is (\cite{H}, Ch. V). Let $C$ be a smooth curve of genus $g \geq 0$ and let $S$ be a ruled surface over $C$. Let $C_0$ be {\it a minimal section} of $S$, i.e. a section of minimal self-intersection. The invariant $e = C_0^2$ is {\it the degree} of $S$ (this differs by sign from Hartshorne's notation). We have:

\begin{lemma} Let $\ms{E}$ be a rank 2 vector bundle on $C$. Consider the ruled surface $S = {\bf P}(\ms{E})$. The following are equivalent:
\begin{enumerate} 
\item $\ms{E}$ is semistable;
\item $e \geq 0$;
\item for any effective divisor $D$ of $S$, we have $D^2 \geq 0$.
\end{enumerate}
\end{lemma}
\begin{proof}
(a) $\Leftrightarrow$ (b) See \cite{H}, Exercise V.2.8.

(b) $\Rightarrow$ (c) Let $C_0$ be a minimal section of $S$. Suppose $ e= C_0^2 \geq 0$. Let $D \sim \mu C_0 - \g{b}f$ be an effective divisor. Then, $\frac 1 2 e \mu - b \geq 0$ where $b = \deg\g{b}$ (this follows from \cite{H}, Prop. V.2.20 if $e=0$ and Prop. 21 if $e > 0$). Hence $D^2 \geq 0$.

(c) $\Rightarrow$ (b) is obvious.
\end{proof}

If $S = {\bf P}(\ms{E})$ with $\ms{E}$ semistable, we will also say that {\it the surface} $S$ is semistable. By the lemma above, $S$ is semistable if and only if $S$ is of degree $e \geq 0$.

\subsection{Elementary Transforms} Let $S$ be a ruled surface over $C$ and let $P$ be a point on $S$. We can create a new ruled surface $S'$ by applying {\it an elementary transform} at $P$ (\cite{H}, Example V.5.7.1). We recall the construction. Denote by $F$ be the fiber of $S$ through $P$. Let $\pi : \wt S \rightarrow S$ be the blowup of $P$ and let $\wt F$ be the strict transform of $F$ in $\wt {S}$. Finally, let $\pi' : \wt S \rightarrow S'$ be the contraction of the (-1)-curve $\wt F$ in $\wt S$ (see Fig. \ref{fig_transform1}).

Similarly, we can define elementary transforms for vector bundles. In the setting above, suppose that $S = {\bf P}(\ms{E})$ where $\ms{E}$ is a rank 2 vector bundle on $C$. Consider the short exact sequence
$$0 \rightarrow \ms{E}' \rightarrow \ms{E} \rightarrow \mathbb{C}(P) \rightarrow 0$$
where the map on the right is just the evaluation map at $P$. The kernel $\ms{E}'$ is again a rank 2 vector bundle on $C$. We can identify $S' = {\bf P}(\ms{E}')$ with the surface constructed above.

The following lemma describes the behavior of semistability under elementary transforms.

\begin{lemma}\label{lemma_ss_fiber} Let $S = {\bf P}(\ms{E})$ be a semistable ruled surface. Let $P$ be a general point on a fixed fiber $F$ of $S$ (the fiber $F$ need not be general). Let $S'$ be the ruled surface obtained from $S$ by applying an elementary transform at $P$. Then $S'$ is semistable, unless $S \cong C \times {\bf P}^1$ is the trivial ruled surface.
\end{lemma}
\begin{proof} If $S'$ is not semistable, there exists a section $C'$ of $S'$ with $(C')^2 < 0$. Denote by $C$ be the strict transform of $C'$ in $S$. Since $S$ is semistable, we have $C^2 \geq 0$. It follows that $C$ passes through $P$ and $C^2 = (C')^2 + 1 = 0$. Since $P$ is a general point on a fiber $F$, it follows that $S \cong C \times {\bf P}^1$.
\end{proof}

\bigskip

\section{First Degeneration}\label{sec_first}

We introduce the first degeneration. The method in this section extends author's previous work in \cite{PE}. As an application we prove Theorem \ref{thm1} below. We are unaware if the result has appeared previously in the literature in this form. 

\begin{theorem}\label{thm1} Let $n$ be a non-square positive integer. Write $n = k^2 + \alpha$ with $k = \lfloor \sqrt{n} \rfloor$. Assume that either: (i) $\alpha$ is even, or (ii) $k \geq 3$. If the linear system $\ms{L}(d,n,m)$ is nonempty, then $d/m \geq c^{(1)}_n$, where
%$$c^{(1)}_n = k + \dfrac { \alpha } { 2k + \dfrac \alpha k }.$$
$$c^{(1)}_n = k + \dfrac { 1 } { (2k/\alpha) + \dfrac 1 k }.$$
\end{theorem}

In particular, the theorem applies for $n=3,6$ and $8$ and any $n \geq 9$.

\begin{example} For $n=3$, we have $c^{(2)}_3 = \frac 3 2$, which is sharp. The linear system $\ms{L}(3,3,2)$ has a unique section, namely the union of 3 lines each passing through 2 points  (\cite{N}, remark on p.772; see also \cite{CM2}, Prop. 2.3).
\end{example}
\begin{example} If $n=6$, we have $c^{(2)}_6 = \frac {12} 5$, which is also sharp. The linear system $\ms{L}(12,6,5)$ has a unique section, namely the union of six conics each passing through 5 of the 6 points  ({\it Loc. cit.}).
\end{example}

\begin{proof}[Proof of the Theorem]
Assume $\ms{L}(d,n,m)$ is nonempty. The idea is to specialize the $n$ general points in ${\bf P}^2$ to a smooth curve $C$ of degree $k$ and then apply Basic Lemma to estimate the multiplicity of vanishing of a general curve in $\ms{L}(d,n,m)$ along $C$.

\smallskip

{\it Step 1.} Let $\Delta$ be the open unit disk over $\mathbb{C}$ and let $X = {\bf P}^2 \times \Delta$. We view $X$ as a relative plane over $\Delta$. For any $t \in \Delta$, denote the fiber $X_t = X \times \{t\}$. Fix a smooth curve $C \subset X_0$ of degree $k$. We have the following split exact sequence
$$0 \rightarrow \ms{N}_{X_0/X}|_C \rightarrow \ms{N}_{C/X} \rightarrow \ms{N}_{C/X_0} \rightarrow 0$$
where $\ms{N}_{X_0/X}|_C \cong \ms{O}_C$ and $\ms{N}_{C/X_0} \cong \ms{O}_C(-kH)$. Consider the ruled surface 
$$S' = {\bf P}(\ms{N}_{C/X}).$$
Let $C'$ be the section of $S'$ corresponding to the short exact sequence above. Note that $C' \sim \ms{O}_{S'}(1)$ and $\ms{O}_{S'}(C') \otimes \ms{O}_{C'}\cong \ms{O}_C(-kH)$.

\smallskip

{\it Step 2.} Choose any set of $n$ distinct points $P_i$ on $C$ (here we do not require the points $P_i \in C$ to be general). Next, we construct a set of $n$ relative points $\ms{P}_i \rightarrow \Delta$ in $X$ specializing to $P_i$ in a general way. Denote by $P'_i$ the images of $\ms{P}_i$ in $S'$. Thus, each $P'_i$ is a general point on the fiber above $P_i$.

Let $\wt X \rightarrow X$ be the blowup of the relative points $\ms{P}_i$ and let $E_i$ denote the corresponding exceptional divisors. Denote by $\wt{C}$ the strict transform of $C$ in $\wt X$. We have the following short exact sequence:
$$0 \rightarrow \ms{N}_{\wt{X}_0/\wt{X}}|_{\wt{C}} \rightarrow \ms{N}_{\wt{C}/\wt{X}} \rightarrow \ms{N}_{\wt{C}/\wt{X}_0} \rightarrow 0$$
We have $\ms{N}_{\wt{X}_0/\wt{X}}|_{\wt{C}} \cong \ms{O}_C$ and $\ms{N}_{\wt{C}/\wt{X}_0} \cong A$, where
$$A = \ms{O}_C(\sum P_i - kH)$$ is a line bundle of degree $\alpha = n - k^2$ on $C$. The short exact sequence corresponds to a certain element
$$\xi \in \op{Ext}^1(A,\ms{O}_C).$$
Next, consider the ruled surface
$$S = {\bf P}(\ms{N}_{\wt C/\wt X}).$$
We identify $C$ with the section of $S$ corresponding to the above exact sequence. Note that $C \sim \ms{O}_S(1)$ and $\ms{O}_S(C) \otimes \ms{O}_C \cong A$.

\smallskip

\begin{figure}[ht]
\includegraphics[scale=.8]{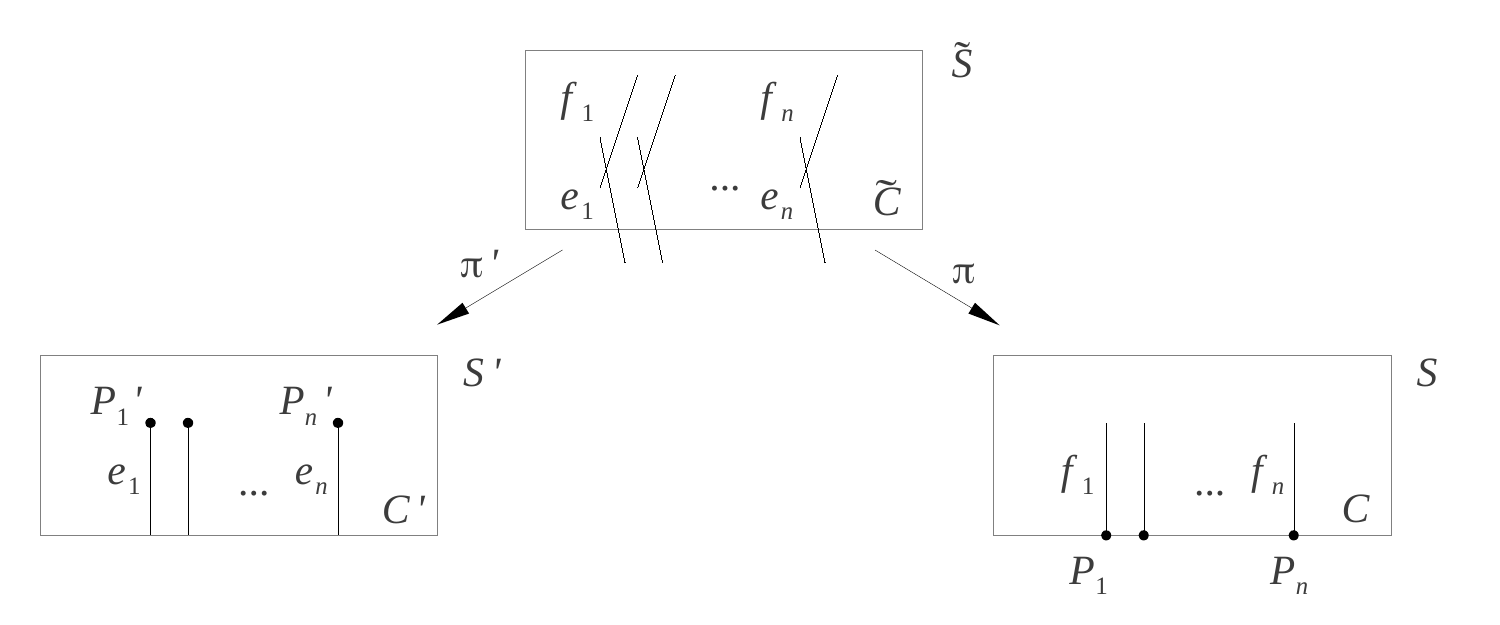}
\caption{Elementary transforms}
\label{fig_transform1}
\end{figure}

The conormal bundles $\ms{N}_{C/X}$ and $\ms{N}_{\wt C/\wt X}$ are related by elementary  transforms at the points $P_i$:
$$0 \rightarrow \ms{N}_{C/X} \rightarrow \ms{N}_{\wt{C}/\wt{X}} \rightarrow \oplus_{i=1}^n {\mathbb C}(P_i) \rightarrow 0$$
Similarly, the ruled surfaces $S'$ and $S$ are related by elementary transforms as on Fig. \ref{fig_transform1} (the intermediate surface $\wt S$ will play a role later in Section \ref{sec_reduction}).

\smallskip

{\it Step 3.} We claim:

\begin{lemma}\label{lemma_xi} Let $\xi \in \op{Ext}^1(A,\ms{O}_C)$ be an arbitrary element. Then $\xi$ can be realized in the above way, for some specialization of points ${\ms P}_i$ to $P_i$.
\end{lemma}

\begin{proof} The idea is to consider the construction in Step 2 in reversed order. Start with any extension
$$0 \rightarrow \ms{O}_C \rightarrow \ms{E} \rightarrow A \rightarrow 0$$
and let $S = {\bf P}(\ms{E})$. As before, we identify $C$ with the section of $S$ corresponding to the above exact sequence. Next, we construct the vector bundle $\ms{E}'$ by applying elementary transforms at the points $P_i$ on $C$:
$$0 \rightarrow \ms{E}' \rightarrow \ms{E} \rightarrow \oplus_{i=1}^n {\mathbb C}(P_i) \rightarrow 0$$
Let $S' = {\bf P}(\ms{E}')$ and let $P'_1,\dots,P'_n$ be as on Fig. \ref{fig_transform1}. It follows that $\ms{E}'$ is realized as an extension
$$0 \rightarrow \ms{O}_C \rightarrow \ms{E}' \rightarrow \ms{O}_C(-kH) \rightarrow 0$$
Now, the key observation is that
$$\op{Ext}^1(\ms{O}_C(-kH),\ms{O}_C) \cong H^1(C, \ms{O}_C(kH)) = 0$$
so the above extension is trivial. This allows us to identify $\ms{E}' \cong \ms{O}_C \oplus \ms{O}_C(-kH)$ with $\ms{N}_{C/X}$, and so $S'$ with ${\bf P}(\ms{N}_{C/X})$. Finally, we choose the relative points $\ms{P}_i$ to pass through $P'_i$ in $S'$. This identifies $\ms{E}$ with $\ms{N}_{\wt C/\wt X}$, and so $S$ with ${\bf P}(\ms{N}_{\wt C/\wt X})$.
\end{proof}

\begin{corollary}\label{lemma_ss} If the specialization of $\ms{P}_i$ to $P_i$ is general enough, the conormal bundle $\ms{N}_{\wt C/\wt X}$ is semistable of slope $\alpha/2$.
\end{corollary}

This follows from Lemma \ref{lemma_xi} and the following general fact:

\begin{lemma} Let $C$ be a curve of genus $g \geq 0$. Let $A$ be a line bundle on $C$ of degree $\alpha \geq 0$. Assume that either: (i) $\alpha$ is even, or (ii) $g \geq 1$. Then, a general element $\xi \in \op{Ext}^1(A,\ms{O}_C)$ corresponds to a semistable rank 2 vector bundle $\ms{E}$ on $C$.
\end{lemma}
\begin{proof} The set of elements $\xi \in \op{Ext}^1(A,\ms{O}_C)$ that correspond to semistable vector bundles $\ms{E}$ is open (this follows from \cite{M2}, Thm. 2.8). So, it suffices to show that the set is nonempty. Now, if $g=0$ and $\alpha = 2\alpha_0$ is even, then $\ms{O}_{{\bf P}^1}(\alpha_0)\oplus \ms{O}_{{\bf P}^1}(\alpha_0)$ is semistable. If $g \geq 1$, one can prove the statement by induction on $\alpha$ by using Lemma \ref{lemma_ss_fiber}. We leave this as an exercise.
\end{proof}

\smallskip

{\it Step 4.} We complete the proof of the theorem. Since $\ms{L}(d,n,m)$ is nonempty by assumption, there is a flat family of curves $\wt{\ms{C}} \rightarrow \Delta$, where $\wt{\ms{C}}$ is a nontrivial section of $|\ms{O}_{\wt X}(dH - \sum m E_i)|$. We are interested in estimating $\mu = \op{mult}_{\wt{C}/\wt{X}}(\wt{\ms{C}})$, which of course is the same as $\mu = \op{mult}_{C/X}(\ms{C})$ where $\ms{C}$ is the image of $\wt{\ms{C}}$ in $X$. Obviously,
$$\mu \leq \frac d k. \eqno(\sharp)$$
By Lemma \ref{lemma_ss} and Cor. \ref{basic_cor}, we have:
$$\underbrace{\op{slope}(\ms{N}_{\wt{C}/\wt{X}})}_{\alpha/2} \cdot \mu \geq \underbrace{- \wt{C} \cdot \wt{\ms{C}} }_{nm - kd}.\eqno(\flat)$$
Combining $(\sharp)$ and $(\flat)$, we get:
$$\frac \alpha 2 \frac d k \geq n m - k d.$$
One can easily check that this is equivalent to the inequality in the theorem. See also Lemma \ref{lemma_B}(a) in the Appendix.
\end{proof}

\bigskip

\section{Reduction of Interpolation Problems to Ruled Surfaces}\label{sec_reduction}

We formalize some results from the previous section. Our result here is Theorem \ref{thm2} which will be used through the rest of the paper.

\begin{notation}{\it A marked surface} $(S;P_1,\dots,P_n)$ is simply a surface $S$ together with $n$ distinct points $P_1,\dots,P_n$ on $S$.
\end{notation}

\begin{notation} Let $S=({\bf P}(\ms{E});P_1,\dots,P_n)$ be a marked ruled surface over $C$. For any integers $(\mu,b,\widehat m)$ and a line bundle $\g{b}$ of degree $b$ on $C$, we denote the line bundle
$$\ms{L}_S(\mu,\g{b},\wh m) = \ms{O}_{\wt S}(\mu - \g{b} f - \sum \wh m e_i)$$
on the blowup $\pi: \wt{S} \rightarrow S$ at the points $P_1,\dots,P_n$, with exceptional divisors $e_1,\dots,e_n$. Here we denote $\ms{O}_{\wt S}(\mu) = \pi^* \ms{O}_S(\mu)$.
\end{notation}

\begin{lemma}\label{lemma_chi} We have
$$\chi(\ms{L}_S(\mu,\g{b},\wh m)) = (\mu+1)(\tfrac \alpha 2 \mu - b - g + 1) - n \tbinom { \wh m + 1 } 2$$
where $\alpha = \deg(\ms{E})$ and $g$ is the genus of $C$.
\end{lemma}
\begin{proof} Denote by $C_1$ the class $c_1(\ms{O}_S(1))$. It follows from (\cite{H}, Lemma V.2.10), that
$$K_S \equiv -2C_1 + (2g-2 + \alpha)f.$$
By the Riemann-Roch formula,
\begin{align*}
\chi(\ms{O}_S(\mu C_1 - \g{b}f)) &= \frac 1 2 (\mu C_1 - \g b f)\cdot (\mu C_1 - \g b f - K_S) + 1 + p_a(S).
\end{align*}
We have $p_a(S) = -g$ (\cite{H}, Cor. V.2.5). The lemma now follows from $C_1^2 = \alpha, C_1 \cdot f = 1$ and $f^2 = 0$.
\end{proof}

\begin{notation} Let $C$ be a smooth curve, $A$ a line bundle on $C$ and let $\xi \in \op{Ext}^1(A,\ms{O}_C)$ corresponding to an extension $$0 \rightarrow \ms{O}_C \rightarrow \ms{E} \rightarrow A \rightarrow 0.$$  We denote by $$S(C,A,\xi)$$ the ruled surface $S = {\bf P}(\ms{E})$ and we identify $C$ with the section determined by the short exact sequence. Note that $C \sim \ms{O}_S(1)$ and $\ms{O}_S(C) \otimes \ms{O}_C \cong A$.
\end{notation}

The following theorem allows to reduce interpolation problems on ${\bf P}^2$ to certain interpolation problems on ruled surfaces.

\begin{theorem}\label{thm2} Let $n$ be a non-square positive integer. Write $n=k^2+\alpha$ with $k \geq 1$ and $\alpha \geq 0$. Consider the linear system $\ms{L}(d,n,m)$ for some positive integers $d$ and $m$. Fix a smooth curve $C$ of degree $k$ in ${\bf P}^2$. Let $S = S(C,A,\xi; \{P_i\})$ be a marked ruled surface where:
\begin{itemize}
\item $A$ is any line bundle of degree $\alpha$ on $C$;
\item $\xi \in \op{Ext}^1(A,\ms{O}_C)$ is any element;
\item $P_1,\dots,P_n$ are distinct points on $C \subset S$ such that $\sum P_i \sim A + kH.$
\end{itemize}
Then, for any $\mu$, we have
$$h^0 (\ms{L}_{{\bf P}^2}(d,n,m)) \leq h^0 (\ms{L}_S(\mu,\g{b},\wh m)) + h^0(\ms{O}_{{\bf P}^2}(\epsilon H)) - h^0(\ms{O}_C(\epsilon H))$$
where:
\begin{itemize}
\item $\g{b} = \ms{O}_C(\sum mP_i - dH)$ of degree $b = nm - kd$;
\item $\wh m = \mu - m$;
\item $\epsilon = d - k \mu$.
\end{itemize}
\end{theorem}

The following lemma justifies our definition of $\ms{L}_S(\mu,\g{b},\wh m)$:
\begin{lemma}\label{lemma_rest} In the setting of the theorem, we have $\ms{L}_S(\mu,\g{b},\widehat m) \otimes \ms{O}_{\wt C} \cong \ms{O}_C(\epsilon H)$, where $\wt{C}$ is the strict transform of $C$ in $\wt S$.
\end{lemma}
\begin{proof} This is an easy computation. We have: $$\ms{O}_S(\mu C - \g{b} f) \otimes \ms{O}_C \cong \ms{O}_C(\mu A - \g b) \cong \ms{O}_C(\sum \wh m P_i + \epsilon H).$$ The lemma follows.
\end{proof}

\begin{proof}[Proof of Theorem] We will use the same construction as in the proof of Theorem \ref{thm1}.

\smallskip

{\it Step 1.} Consider the threefold $X = {\bf P}^2 \times \Delta$. We identify $C$ with a curve on $X_0$. Given \mbox{$\xi \in \op{Ext}^1(A,\ms{O}_C)$}, we specialize the $n$ relative points $\ms{P}_i$ to $P_i \in C$ as in Lemma \ref{lemma_xi}. As before, let $\wt{X} \rightarrow X$ be the blowup of the $\ms{P}_i$, and let $\wt C$ be the strict transform of $C$ in $\wt{X}$. We have the short exact sequence
$$0 \rightarrow \ms{N}_{\wt{X}_0/\wt{X}}|_{\wt{C}} \rightarrow \ms{N}_{\wt{C}/\wt{X}} \rightarrow \ms{N}_{\wt{C}/\wt{X}_0} \rightarrow 0$$
where $\ms{N}_{\wt{X}_0/\wt{X}}|_{\wt{C}} \cong \ms{O}_C$ and $\ms{N}_{\wt{C}/\wt{X}_0} \cong A$. By construction, the above extension corresponds to $\xi$. In particular, $S \cong {\bf P}(\ms{N}_{\wt C/\wt X})$, where $S = S(C,A,\xi)$ is the ruled surface we started with.

\smallskip

{\it Step 2.} Consider the threefold $Y$ obtained from $X = {\bf P}^2 \times \Delta$ by first blowing up $C$ (with exceptional divisor $S' = {\bf P}(\ms{N}_{C/X})$), followed by blowing up the strict transforms of the relative points $\ms{P}_1,\dots,\ms{P}_n$ (with exceptional divisors $E_1,\dots,E_n$). We view $Y \rightarrow \Delta$ as a flat family with general fiber $Y_t \cong \wt{X}_t$. The special fiber $Y_0$ is the union of two surfaces $\wt{S} \cup X_0$ meeting transversely along $\wt C \cong C$. \footnote{We denote by $\wt{C}$, resp. $C$, the same curve in $Y$ viewed as a divisor in $\wt S$, resp. $X_0$.}

This construction is related to the construction in Step 1 as follows. Let $\wt X'$ be the threefold obtained from $\wt{X}$ by blowing up $\wt{C}$ (with exceptional divisor $S = {\bf P}(\ms{N}_{\wt C/\wt X})$). Then, $Y$ can be obtained from $\wt{X}'$ by applying (-1)-transfers to the exceptional curves $e_1,\dots,e_n$ as on Fig. \ref{fig_transform2} (see also \cite{CM1}, Section 4.1). The induced map $\pi : \wt{S} \rightarrow S$ coincides with the corresponding map on Fig. \ref{fig_transform1}.

\begin{figure}[ht]
\includegraphics[scale=.8]{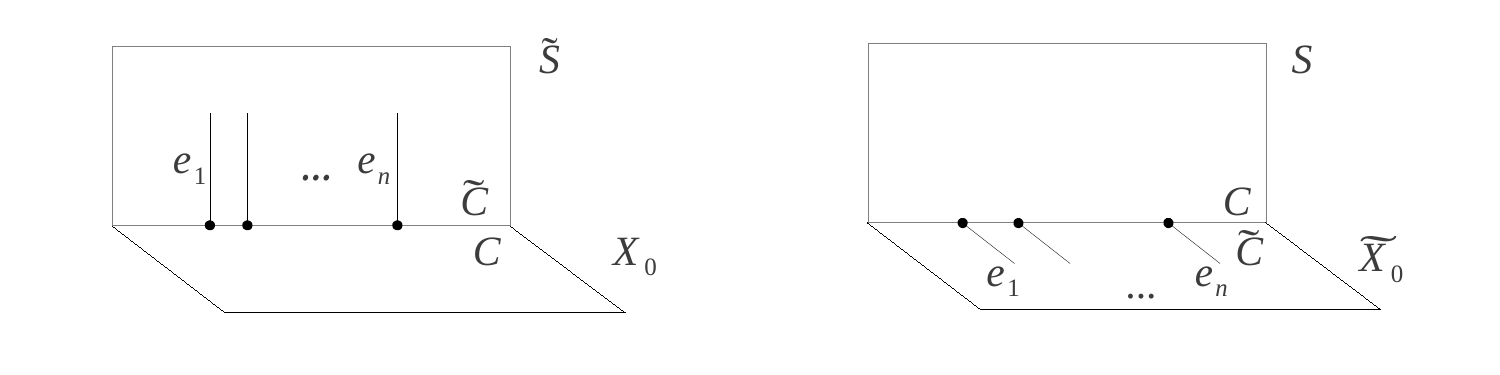}
\caption{(-1) transfers between $Y$ and $\wt X'$}
\label{fig_transform2}
\end{figure}

\smallskip

{\it Step 3.} For a given $\mu$, consider the following line bundle on $Y$:
$$\ms{L}_Y \cong \ms{O}_Y(dH - \sum m E_i - \mu \wt S).$$
We view $\ms{L}_Y$  as a flat family of line bundles with general fiber $\ms{L}_{Y_t} \cong \ms{O}_{\wt X_t} (dH - \sum m E_i)$. The special fiber $\ms{L}_{Y_0}$ is described by the following short exact sequence:
$$0 \rightarrow \ms{L}_{Y_0} \rightarrow \ms{L}_{\wt{S}} \oplus \ms{L}_{X_0} \rightarrow \ms{L}_C \rightarrow 0 \eqno(*)$$
where
\begin{align*}
\ms{L}_{\wt S} \cong \ms{L}_{Y} \otimes \ms{O}_{\wt S}; \ \ \ms{L}_{X_0} \cong \ms{O}_{{\bf P}^2}(\epsilon H); \ \ \ms{L}_C \cong \ms{O}_C(\epsilon H). \ \ \ \ 
\end{align*}

\begin{lemma} We have: $$\ms{L}_{\wt S} \cong \ms{L}_S(\mu,\g{b},\widehat m).$$
\end{lemma}
\begin{proof} The line bundle $\ms{L}_{\wt S}$ has the following properties:
$$\ms{L}_{\wt S} \otimes \ms{O}_{\wt C} \cong \ms{O}_C(\epsilon H); \ \ \ms{L}_{\wt S} \cdot f = \mu; \ \ \ms{L}_{\wt S} \cdot e_i = \wh m.$$
From Lemma \ref{lemma_rest}, $\ms{L}_S(\mu,\g{b},\widehat m)$ has the same properties. It follows that $\ms{L}_{\wt S} \cong \ms{L}_S(\mu,\g{b},\widehat m)$.
\end{proof}

To complete the proof of the theorem, take cohomology in $(*)$:
$$0 \rightarrow H^0(\ms{L}_{Y_0}) \rightarrow H^0(\ms{L}_{\wt S}) \oplus H^0(\ms{O}_{{\bf P}^2}(\epsilon H)) \rightarrow H^0(\ms{O}_C(\epsilon H)) \mathop {\rightarrow}^\delta H^1(\ms{L}_{Y_0})$$
The restriction $H^0(\ms{O}_{{\bf P}^2}(\epsilon H)) \rightarrow H^0(\ms{O}_C(\epsilon H))$ is surjective. Hence, the coboundary map $\delta = 0$. The theorem now follows from the semicontinuity principle applied to $h^0(\ms{L}_{Y_t})$.
\end{proof}

\begin{corollary}  Assume the above setting.
\begin{enumerate}
\item For any $\mu$, we have:
$$\chi(\ms{L}_{{\bf P}^2}(d,n,m)) = \chi(\ms{L}_S(\mu,\g{b},\wh m)) + \chi(\ms{O}_{{\bf P}^2}(\epsilon H)) - \chi(\ms{O}_C(\epsilon H)).$$
\item For any $\mu \leq d/k$ (hence $\epsilon \geq 0$), we have:
$$\chi(\ms{L}_{{\bf P}^2}(d,n,m)) = \chi(\ms{L}_S(\mu,\g{b},\wh m)) + h^1(\ms{O}_C(\epsilon H)).$$
\end{enumerate}
\end{corollary}
\begin{proof} (a) This follows from the short exact sequence $(*)$ and the fact that $\chi(\ms{L}_{Y_t})$ is a constant function of $t$.

(b) This follows from (a).
\end{proof}

\begin{corollary}\label{cor_thm2} Suppose the linear system $\ms{L}(d,n,m)$ is nonempty. Then $|\ms{L}_S(\mu,\g{b},\wh m)|$ is nonempty with $\mu = \lfloor d/k \rfloor$.
\end{corollary}
\begin{proof} %It suffices to take $\mu$ as in the proof of Theorem \ref{thmA}.
Consider the long exact cohomology sequence associated to $(*)$. If $\mu = \lfloor d/k \rfloor$, the restriction $H^0(\ms{O}_{{\bf P}^2}(\epsilon H)) \rightarrow H^0(\ms{O}_C(\epsilon H))$ is an isomorphism. The claim follows.
\end{proof}

\bigskip

\section{Families of Ruled Surfaces}\label{sec_family}

We can use the degeneration technique from the previous section to reduce an interpolation problem on ${\bf P}^2$ to an interpolation problem on a certain ruled surface $S$. We would like to perform further degenerations to study the later problem. Our first goal is to define an object $\g{S}(Z,\ms{A},\xi) \rightarrow \Delta$ which is a relative analogue of $S(C,A,\xi)$. We conclude with a technical result (Prop. \ref{prop_smooth}) which will be used in Sections \ref{sec_eleven} and \ref{sec_last}. As usual, $\Delta$ denotes the open unit disk over ${\mathbb C}$.

\begin{notation} Let $\ms{A}$ be a torsion-free sheaf of rank 1 on $C \times \Delta$. Assume that $\ms{A} = \ms{A}' \otimes \ms{I}_W$ where $\ms{A}'$ is invertible and $\ms{I}_W$ is the ideal sheaf of a l.c.i. zero-dimensional subscheme $W \subset C \times \Delta$ (possibly $W=\emptyset$). Let $\xi \in \op{Ext}^1(\ms{A},\ms{O}_{C\times \Delta})$ be an element corresponding to an extension $$0 \rightarrow \ms{O}_{C \times \Delta} \rightarrow \ms{E} \rightarrow \ms{A} \rightarrow 0$$
such that $\ms{E}$ is locally free. We denote by $$\g{S}(Z,\ms{A},\xi) \rightarrow \Delta$$ the relative ruled surface $\g{S} = {\bf P}(\ms{E}) \rightarrow \Delta$ together with the subscheme $Z = {\bf P}(\ms{A})$ defined by the above exact sequence. In particular, $Z$ is a divisor of $\g{S}$ with $Z \sim \ms{O}_{\g S}(1)$. 
\end{notation}

We will assume that $W$ is supported on $C \times \{0\}$. In applications, $W$ will be reduced; however, everything we say in this section holds in the more general setting. 

\begin{lemma} In the above setting, the projection $p : Z \rightarrow C \times \Delta$ is just the blowup of $W \subset C \times \Delta$. Denote by $F$ the exceptional divisor of $p$. Then $p^* \ms{A} \cong p^* \ms{A}' \otimes \ms{O}_{Z}(-F) \cong \ms{O}_{\g{S}}(Z) \otimes \ms{O}_{Z}$.
\end{lemma}

\begin{proof} Since $\ms{A}'$ is invertible, $Z = {\bf P}(\ms{A}) = {\bf P}(\ms{I}_W)$ which is exactly the definition of a blowup. The rest is clear.
\end{proof}

Consider the following question: given $\ms{A}$ as above, which elements $\xi \in \op{Ext}^1(\ms{A},\ms{O}_{C\times\Delta})$ correspond to a locally free extension $\ms{E}$? Clearly, if $W = \emptyset$, then $\ms{A}$ is locally free and so any $\xi$ will do. In the general case, we have:

\begin{proposition}\label{prop_serre} Let $\ms{A} = \ms{A}' \otimes \ms{I}_W$ be as above. Then, there is a natural exact sequence
$$0 \rightarrow H^1(C \times \Delta, (\ms{A}')^{-1}) \rightarrow \op{Ext}^1(\ms{A},\ms{O}_{C\times\Delta}) \mathop\rightarrow^{\delta} H^0(\ms{E}xt^1(\ms{A},\ms{O}_{C\times \Delta})) \rightarrow 0$$

Moreover, we have:

(a) $\ms{E}xt^1(\ms{A},\ms{O}_{C\times\Delta}) \cong \ms{E}xt^2(\ms{O}_W,\ms{O}_{C\times\Delta}) \cong \ms{O}_W$.

(b) The extension $\ms{E}$ corresponding to $\xi \in \op{Ext}^1(\ms{A},\ms{O}_{C\times \Delta})$ is locally free if and only if $\delta(\xi)$ generates the sheaf $\ms{E}xt^1(\ms{A},\ms{O}_{C\times\Delta})$.
\end{proposition}

\begin{proof} See \cite{F}, Chapter 2, p. 36--37. The exact sequence follows from the local-to-global spectral sequence
$$E^{i,j}_2 = H^i(\ms{E}xt^j(\ms{A},\ms{O}_{C\times\Delta})) \Longrightarrow \op{Ext}^{i+j}(\ms{A},\ms{O}_{C\times\Delta}).$$
Note that $E^{i,0}_2 = H^i(\ms{A}^{-1}) = H^i((\ms{A}')^{-1}) = H^0(\Delta,R^i \pi_* (\ms{A}')^{-1})$ where $\pi: C\times\Delta \rightarrow \Delta$ is the projection. In particular, $E^{i,0}_2 = 0$ for $i \geq 2$.

(a) {\it Ibid.}, Lemma 7 (note that the isomorphisms are not canonical).

(b) {\it Ibid.}, Theorem 8.
\end{proof}

Consider a relative ruled surface $\g{S}(Z,\ms{A},\xi) \rightarrow \Delta$. For any $t \in \Delta$, $\g{S}_t = {\bf P}(\ms{E}_t)$ is a ruled surface over $C \times \{t\}$ where $\ms{E}_t$ arises as an extension
$$0 \rightarrow \ms{O}_{C\times\{t\}} \rightarrow \ms{E}_t \rightarrow \ms{A}_t \rightarrow 0$$
We have $Z_t \sim \ms{O}_{\g{S}_t}(1)$. For a general $t$, the subscheme $Z_t$ is a section of $\g{S}_t$. On the special fiber, we have $Z_0 = C_0 \cup F_0$, where $C_0$ is a section of $\g{S}_0$ and $F_0 = F \times \{0\}$ is the vertical component of $Z_0$. If $F_0 \neq \emptyset$, we will say that $Z_0$ is a {\it degenerate section} of $\g{S}_0$.

Next, we will show that any degenerate section can be smoothed, in the following sense.

\begin{proposition}\label{prop_smooth} Let $\ms{A} = \ms{A}' \otimes \ms{I}_W$ be as above, with $W$ is supported on $C \times \{0\}$. Let
$$0 \rightarrow \ms{O}_{C\times\{0\}} \rightarrow \ms{E}_0 \rightarrow \ms{A}_0 \rightarrow 0$$
be any extension, with $\ms{E}_0$ locally free. Then, the exact sequence can be extended to 
$$0 \rightarrow \ms{O}_{C\times \Delta} \rightarrow \ms{E} \rightarrow \ms{A} \rightarrow 0$$
with $\ms{E}$ locally free on $C \times \Delta$.
\end{proposition}

\begin{proof} Consider the commutative diagram with exact rows:
$$
\begin{array}{ccccccc}
0 \rightarrow & H^1(C\times\Delta, (\ms{A}')^{-1}) & \rightarrow & \op{Ext}^1(\ms{A},\ms{O}_{C\times\Delta}) & \rightarrow & H^0(\ms{E}xt^1(\ms{A},\ms{O}_{C\times \Delta})) & \rightarrow 0 \\
              & \downarrow         &             & \downarrow                            &             & \downarrow                              &               \\
0 \rightarrow & H^1(C \times \{0\},(\ms{A}_0)^{-1}) & \rightarrow & \op{Ext}^1(\ms{A}_0,\ms{O}_{C\times \{0\}}) & \rightarrow & H^0(\ms{E}xt^1(\ms{A}_0,\ms{O}_{C\times \{0\}})) & \rightarrow 0
\end{array}
$$
\\
Note that $\ms{A}_0 \cong \ms{A}'_0(-W_0) \oplus \ms{O}_{W_0}$ where $W_0 = W \times \{0\}$. Therefore, $(\ms{A}_0)^{-1} \cong (\ms{A}'_0(-W_0))^{-1}$. It follows that the bottom row of the diagram splits. Now, the map on the left factors through 

$$\underbrace{H^1(C\times\Delta, (\ms{A}')^{-1})}_{H^0(\Delta,R^1\pi_*(\ms{A}')^{-1})} \rightarrow \underbrace{H^1(C\times\{0\}, (\ms{A}'_0)^{-1})}_{(R^1\pi_*(\ms{A}')^{-1})_0} \rightarrow H^1(C \times \{0\}, (\ms{A}_0)^{-1})$$ 
\\
which is surjective. The map on the right is just the restriction
$$H^0(\ms{O}_W) \rightarrow H^0(\ms{O}_{W_0})$$
which is also surjective. By the Short Five Lemma, the map in the middle is surjective as well. Hence, any given extension $\xi_0 \in \op{Ext}^1(\ms{A}_0,\ms{O}_{C \times \{0\}})$ can be lifted to $\xi \in \op{Ext}^1(\ms{A},\ms{O}_{C\times\Delta})$. The resulting $\ms{E}$ is locally free by Prop. \ref{prop_serre}(b).
\end{proof}

\bigskip

\section{Main Result -- Overview}\label{sec_main}

The following theorem was announced in the introduction. The proof will occupy the rest of the paper. We will consider a certain refinement in Section \ref{sec_refinement}.

\begin{main_theorem} Let $n$ be a non-square positive integer. Write $n=k^2+\alpha$ with $k=\lfloor \sqrt{n} \rfloor$. Assume that either:
\begin{enumerate}[i)]
\item $n=8,10,12$, or
\item $k\geq 3$, $\alpha$ is even, $\alpha \mid 2n$.
\end{enumerate}
If the linear system $\ms{L}(d,n,m)$ is nonempty, then $d/m \geq c^{(2)}_n$.
\end{main_theorem}

\smallskip

The proof of the theorem consists of the four steps outlined below.

\subsection{Setup} We assume $\ms{L}(d,n,m)$ is nonempty. Fix a smooth curve $C$ of degree $k$ in ${\bf P}^2$. By Cor. \ref{cor_thm2}, the linear system $|\ms{L}_{S}(\mu,\mathfrak{b},\wh m)|$ is nonempty, with $\mu = \lfloor d/k \rfloor$, for any marked ruled surface $S(C,A,\xi;\{P_i\})$ as in Theorem \ref{thm2}.

\subsection{Degeneration} We will construct a relative marked ruled surface $\g{S}(Z,\ms{A},\xi;\{\ms{P}_i\})$ over $\Delta$ such that the general fiber $\g{S}_t$ satisfies the assumptions of Theorem \ref{thm2}, i.e.:

\begin{itemize}

\item $\ms{A}$ is a line bundle on $C \times \Delta$ of relative degree $\alpha = \deg \ms{A}_t$.

\item $\ms{P}_1,\dots,\ms{P}_n$ lie on $Z$. We denote by $\overline{\ms{P}}_i$ the projection of $\ms{P}_i$ under $Z \rightarrow C \times \Delta$.

\item For $t \in \Delta$ general, we have $\sum \ms{P}_{i,t} \sim \ms{A}_t + kH$ on $Z_t \cong C$.
\end{itemize}

We now describe the special fiber $\g{S}_0$. We assume that $\g{S}_0$ is semistable. Next, we assume that there is a smooth (possibly disconnected) curve $\Gamma$ on $\g{S}_0$ with the following two properties:

\begin{itemize}
\item $\Gamma$ meets $Z_0$ transversely at $n$ distinct points points $P_i = \ms{P}_{i,0}$;
\item $\Gamma^2 = 0$.
\end{itemize}

This determines uniquely the numerical class of $\Gamma$:
$$\Gamma \equiv \frac {2n}\alpha(Z_0 - \frac \alpha 2 f)$$
where $Z_0 \sim \ms{O}_{\g{S}_0}(1)$. In particular, a necessary condition for the existence of $\Gamma$ is that $\alpha \mid 2n$.

Let $\Gamma = \sum_{i=1}^s \Gamma_i$ where each $\Gamma_i$ is a smooth irreducible curve, with $\Gamma_i \cdot \Gamma_j = 0$ for $i \neq j$. Since $\Gamma^2 = 0$ and $\g{S}$ is semistable, $\Gamma$ lies on the boundary of the effective cone of $\g{S}_0$ (this follows from Lemma \ref{lemma_ss_fiber}). Therefore, $\Gamma_i \equiv \lambda_i \Gamma$ for some $\lambda_i \in \mathbb{Q}$ with $\sum \lambda_i = 1$ (in fact, in applications we will always have $\lambda_1 = \dots = \lambda_s = 1/s$).

\subsection{Semistability} Denote by $\pi : \wt{\g{S}} \rightarrow \g{S}$ the blowup of $\ms{P}_1,\dots,\ms{P}_n$ and let $E_1,\dots,E_n$ be the corresponding exceptional divisors. Denote by $\wt{\Gamma}$ the strict transform of $\Gamma$. We make the following hypothesis:
\begin{itemize}
\item for each $i$, the conormal bundle $\ms{N}_{\wt{\Gamma}_i/\wt{\g{S}}}$ is semistable of slope $\frac 1 2 \lambda_i n$.
\end{itemize}

\subsection{Invariants} Assuming the construction above can be realized, we complete the proof of the theorem. Consider the line bundle
$$\ms{L}_{\g{S}}(\mu,\ms{B},\wh m) = \ms{O}_{\wt{\mathfrak S}}(\mu - \ms{B} f - \sum \wh m E_i)$$
on the blowup $\pi : \wt{\g{S}} \rightarrow \g{S}$, where
\begin{itemize}
\item $\mu = \lfloor d /k \rfloor$;
\item $\ms{B} = \ms{O}_{C\times \Delta}(\sum m \overline{\ms{P}}_i - dH)$ of relative degree $b = \deg{\ms{B}}_t = nm - kd$.
\item $\wh m = \mu - m$.
\end{itemize}

By construction, there is a flat family of curves $\wt{\ms{C}} \rightarrow \Delta$ in $\wt{\g{S}}$, where $\wt{\ms{C}}$ is a section of $|\ms{L}_{\g{S}}(\mu,\ms{B},\wh m)|$. Denote by $\ms{C}$ the projection of $\wt{\ms{C}}$ in $\g{S}$. The following is a key computation:

\begin{lemma} Define $\gamma = -\frac 1 n \wt \Gamma \cdot \wt {\ms{C}}.$ Then
$$\gamma = \frac{n+k^2}\alpha m - \frac {2k}\alpha d.$$
\end{lemma}
\begin{proof}
Using the fact that $\Gamma \cdot Z_0 = n$, $\Gamma \cdot f = \frac{2n}\alpha$ and $\ms{C}_0 \sim \mu Z_0 - \g{b} f$, we find:
\begin{align*}
\gamma &= - \frac 1 n \wt \Gamma \cdot \wt{\ms{C}} \\
&= -\frac 1 n \left(\Gamma \cdot \ms{C} - n\wh m \right) \\
&= -(\mu - \frac 2 \alpha b) + \wh m \\
&= -m + \frac 2 \alpha b.
\end{align*}
Finally, substitute $b=nm-kd$ and $n=k^2+\alpha$.
\end{proof}

To complete the proof of the theorem, we will estimate $\mu_i = \op{mult}_{\Gamma_i/\g{S}}(\ms{C})$ in two ways. First, there is an obvious upper bound which comes from the numerical class of $\Gamma$:
$$\sum \lambda_i \mu_i \leq \frac {\mu} {2n/\alpha}.$$
Since $\mu \leq d/k$, this becomes:
$$\sum \lambda_i \mu_i \leq \frac{d}{2kn/\alpha}. \eqno(\sharp)$$
By Basic Lemma and the semistability hypothesis, we have:
$$\underbrace{\op{slope}(\ms{N}_{\wt \Gamma_i/\wt{\mathfrak{S}}})}_{\frac 1 2 \lambda_i n} \cdot \mu_i \geq \underbrace{-\wt \Gamma_i \cdot \wt{\ms{C}}}_{\lambda_i n\gamma}$$
i.e.
$$\frac 1 2 \mu_i \geq \gamma. \eqno(\flat)$$
From $(\sharp)$ and $(\flat)$, and using $\sum \lambda_i = 1$, we get:
$$\frac 1 2 \frac{d}{2kn/\alpha} \geq \gamma = \frac{n+k^2}{\alpha} m - \frac{2k}\alpha d.$$
It turns out that this is equivalent to the inequality in the theorem. We will check this explicitly for specific values of $n$. For the general case, see Lemma \ref{lemma_B}(b) in the Appendix.

\bigskip

\section{Eight Points}\label{sec_eight}

We verify Main Theorem in the case $n=8$. The value $c^{(2)}_8 = \frac{48}{17}$ is well-known to be sharp (see example below). We include this case for illustration purposes.

\subsection{Setup} Since $k=2$, we take $C$ to be a smooth conic in ${\bf P}^2$. The line bundle $A \cong \ms{O}_{{\bf P}^1}(4)$ on $C$ is of degree $\alpha=4$. Consider an extension
$$0 \rightarrow \ms{O}_C \rightarrow \ms{E} \rightarrow A \rightarrow 0$$
corresponding to a general $\xi \in \op{Ext}^1(A,\ms{O}_C)$. It follows that
$$\ms{E} \cong \ms{O}_{{\bf P}^1} (2) \oplus \ms{O}_{{\bf P}^1}(2).$$
Hence, $S = {\bf P}(\ms{E}) \cong C \times {\bf P}^1 \cong {\bf P}^1 \times {\bf P}^1$. We identify $C$ with the section of $S$ corresponding to the short exact sequence. It follows that $C \sim C_0+2f$ where $C_0$ is a horizontal section of $S$.

\subsection{Degeneration} Let $\g{S} = S \times \Delta$ and $Z = C \times \Delta \subset \g{S}$. We identify $S$ with the special fiber of $\g{S} \rightarrow \Delta$. We take $\Gamma = \Gamma_1 + \dots + \Gamma_4$ on $S$, where each $\Gamma_i \sim C_0$ is a general horizontal section of $S$. Let $\Gamma_i \cap C = \{P_{2i-1},P_{2i}\}$. Next we specialize the eight relative points $\ms{P}_{i} \subset Z$ to $P_i \in C$ in a general way.

\subsection{Semistability} Let $\wt{\g{S}}$ be the blowup of $\g{S}$ along the relative points $\ms{P}_i$. We have to show that, for each $i$, the conormal bundle $\ms{N}_{\wt \Gamma_i/\wt{\g{S}}}$ is semistable (of slope 1). Denote by $P'_i$ the image of $\ms{P}_i$ in the corresponding ${\bf P}(\ms{N}_{\Gamma/\g{S}}) \cong \Gamma \times {\bf P}^1$. Then, ${\bf P}(\ms{N}_{\wt \Gamma_i/\wt{\g{S}}})$ is obtained from ${\bf P}(\ms{N}_{\Gamma_i/\g{S}})$ by performing elementary transforms at the points $P_{2i-1}',P_{2i}'$. If the specialization is general enough, the points $P_{2i-1}',P_{2i}'$ do not belong to the same horizontal section of ${\bf P}(\ms{N}_{\Gamma_i/\g{S}})$. It follows that $\ms{N}_{\wt \Gamma_i/\wt{\g{S}}} \cong \ms{O}_{\Gamma_i}(P_{2i-1}) \oplus \ms{O}_{\Gamma_i}(P_{2i}) \cong \ms{O}_{{\bf P}^1}(1) \oplus \ms{O}_{{\bf P}^1}(1)$, which is semistable of slope 1.

\subsection{Invariants}
Denote $\mu_i = \op{mult}_{\Gamma_i/\g{S}}(\ms{C})$. By symmetry, $\mu_1 = \dots = \mu_4$. Since $\Gamma_i \sim C_0$, we have the following upper bound:
$$\mu_1 \leq \frac \mu { 4 } \leq \frac {d} { 4 \cdot 2 }. \eqno(\sharp)$$
The lower bound from Basic Lemma is:
$$\frac 1 2 \mu_1 \geq \gamma = 3m - d. \eqno(\flat)$$
Combining $(\sharp)$ and $(\flat)$, we get:
$$\frac 1 2 \cdot \frac {d} { 8 } \geq 3 m - d \Longleftrightarrow  17 d \geq 48m,$$
q.e.d.

The bound is sharp:
\begin{example} The linear system $\ms{L}(48,8,17)$ has a unique section, namely the union of 8 curves of degree 6 each passing through 7 points with multiplicity 2 and 1 point with multiplicity 3 (\mbox{\cite{N}, remark on p. 772}; \cite{CM2}, Prop. 2.3).
\end{example}

\bigskip

\section{Ten Points}\label{sec_ten}

Here we prove Main Theorem for $n=10$ points.

\subsection{Setup} Since $k=3$, we take $C$ to be a smooth cubic in ${\bf P}^2$. Fix a point $W$ on $C$ such that
$$9W \sim 3H.$$
For example, we can take $W$ to be a Weierstrass point of $C$ (however, later in Section \ref{sec_refinement} we will require that $3W \nsim H$). Let $A = \ms{O}_C(W)$ which is of degree $\alpha=1$. Since $h^1(A^\vee) = 1$, there is a unique nontrivial extension
$$0\rightarrow \ms{O}_C \rightarrow \ms{E} \rightarrow A \rightarrow 0$$
The surface $S = {\bf P}(\ms{E})$ is an indecomposable elliptic ruled surface of degree 1 (see Appendix \ref{appendix_A} for background). The short exact sequence determines a minimal section of $S$ which we identify with the curve $C$.

\subsection{Degeneration} Let $\g{S} = S \times \Delta$ and $Z = C \times \Delta \subset \g{S}$. We identify $S$ with the special fiber of $\g{S} \rightarrow \Delta$. Next, we take $\Gamma = \Gamma_1 + \dots + \Gamma_5$, where each $\Gamma_i$ is a general section of the pencil $|-2K_S|$. By Prop. \ref{prop_elliptic}, each $\Gamma_i \equiv 4C - 2f$ is a smooth elliptic curve. Denote $\Gamma_i \cap C = \{P_{2i-1},P_{2i}\}$. Note that $P_{2i-1} + P_{2i} \sim 2A.$ Since $3A \sim H$, we have
$$\sum P_i \sim 10A \sim A + 3H.$$
Finally, we specialize the ten relative points $\ms{P}_1,\dots,\ms{P}_{10}$ in $Z$ to $P_1,\dots,P_{10}$ in a general way such that
$$\sum \ms{P}_{i,t} \sim A + 3H$$
for any $t \in \Delta$.

\subsection{Semistability} We claim that, for each $i$, $\ms{N}_{\wt \Gamma_i/\wt{\g{S}}}$ is semistable (of slope 1). Denote by $P'_i$ the image of $\ms{P}_i$ in ${\bf P}(\ms{N}_{\Gamma/\g{S}}) \cong {\bf P}(\ms{O}_\Gamma \oplus \ms{O}_\Gamma) \cong \Gamma \times {\bf P}^1$. Now, ${\bf P}(\ms{N}_{\wt \Gamma_i/\wt{\g{S}}})$ is obtained from ${\bf P}(\ms{N}_{\Gamma_i/\g{S}})$ by performing elementary transforms at the points $P_{2i-1}',P_{2i}'$. If the specialization is general enough, the points $P_{2i-1}',P_{2i}'$ do not belong to the same horizontal section of ${\bf P}(\ms{N}_{\Gamma_i/\g{S}})$. It follows that $\ms{N}_{\wt \Gamma_i/\wt{\g{S}}} \cong \ms{O}_{\Gamma_i}(P_{2i-1}) \oplus \ms{O}_{\Gamma_i}(P_{2i})$, which is semistable of slope 1.

\begin{figure}[ht]
\includegraphics[scale=.8]{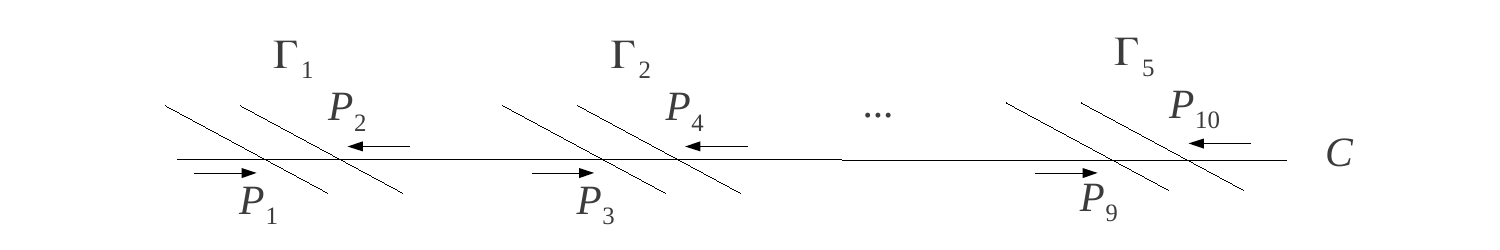}
\caption{Specialization to $\Gamma$ $(n=10)$}
\label{ten_fig}
\end{figure}

\subsection{Invariants} Denote $\mu_i = \op{mult}_{\Gamma_i/\g{S}}(\ms{C})$. By symmetry, $\mu_1 = \dots = \mu_5$. Since $\Gamma_i \equiv 4C - 2f$, we have the following upper bound:
$$\mu_1 \leq \frac \mu { 5 \cdot 4 } \leq \frac {d} { 5 \cdot 4 \cdot 3 }. \eqno(\sharp)$$
The lower bound from Basic Lemma is:
$$\frac 1 2 \mu_1 \geq \gamma = 19m - 6d. \eqno(\flat)$$
From $(\sharp)$ and $(\flat)$, we get:
$$\frac 1 2 \cdot \frac {d} { 60 } \geq 19 m - 6 d \Longleftrightarrow  721 d \geq 2280m,$$
q.e.d.

\bigskip

\section{Eleven Points}\label{sec_eleven}

We prove Main Theorem for $n=11$ points. This is the first time when we study an interpolation problem $\ms{L}_S(\mu,\g{b},\wh m)$ by deforming the underlying surface $S$ itself.

\subsection{Setup} As before, $C$ is a smooth cubic in ${\bf P}^2$. Fix a point $W$ on $C$ such that $9W \sim 3H.$
Let $A$ be any line bundle of degree $\alpha=2$ on $C$. It is easy to see that for a general $\xi \in \op{Ext}^1(A,\ms{O}_C)$, the ruled surface $S(C,A,\xi)$ is decomposable of degree 0. Denote by $C_{(i)}$, $i=0,1$, the two minimal sections of $S$. It follows that $C \equiv C_{(i)} + f$.

\subsection{Degeneration} We will construct a relative marked ruled surface $\g{S}(Z,\ms{A},\xi;\{\ms{P}_{ij}\})$ such that:

\begin{itemize}
\item The special fiber $\g{S}_0$ is simply $C \times {\bf P}^1$.
\item The special section $Z_0 = C_0 \cup F$; here $C_0$ is a horizontal section of $\g{S}_0$ and $F$ is the fiber of $\g{S}_0$ above $W$.
\item The relative points $\{\ms{P}_i\}$ on $Z$ are such that $\sum \ms{P}_{i,t} \sim \ms{A}_t + kH$ on $Z_t \cong C$, for general $t$.
\item Each limit point $P_{i} = \ms{P}_{i,0}$ is a general point on $F$.
\end{itemize}

The construction is done as follows. First, we choose relative points $\overline{\ms{P}}_i$ in $C \times \Delta$ specializing to $W \times \{0\}$ in a general way. Let
\begin{align*} \ms{A}' &= \ms{O}_{C\times\Delta}(\sum \overline{\ms{P}}_i - 3H); \\
\ms{A} &= \ms{A}' \otimes \ms{I}_{W\times \{0\}}.
\end{align*}
Since $9W \sim 3H$, it follows that
\begin{align*}
\ms{A}'_0 &\cong \ms{O}_C(11W - 3H) \cong \ms{O}_{C}(2W); \\
\ms{A}_0 &\cong \ms{O}_C(W) \oplus \ms{O}_W(W).
\end{align*}
Consider the following short exact sequence on $C \times \{0\}$:
$$0 \rightarrow \ms{O}_C \rightarrow \ms{O}_C(W) \oplus \ms{O}_C(W) \rightarrow \ms{O}_C(W) \oplus \ms{O}_W(W) \rightarrow 0$$
By Prop. \ref{prop_smooth}, the sequence can be extended to 
$$0 \rightarrow \ms{O}_{C\times\Delta} \rightarrow \ms{E} \rightarrow \ms{A} \rightarrow 0$$
with $\ms{E}$ locally free. Let $\g{S} = {\bf P}(\ms{E})$ and $Z = {\bf P}(\ms{A})$. Hence $\g{S}_0 = {\bf P}(\ms{O}_C(W) \oplus \ms{O}_C(W)) = C \times {\bf P}^1$. Now, $Z \rightarrow C \times \Delta$ is just the blowup of $W \times \{0\}$ with exceptional divisor $F$. Finally, we take $\ms{P}_i$ to be the strict transform of $\overline{\ms{P}}_i$ in $Z$.

\begin{figure}[ht]
\includegraphics[scale=.8]{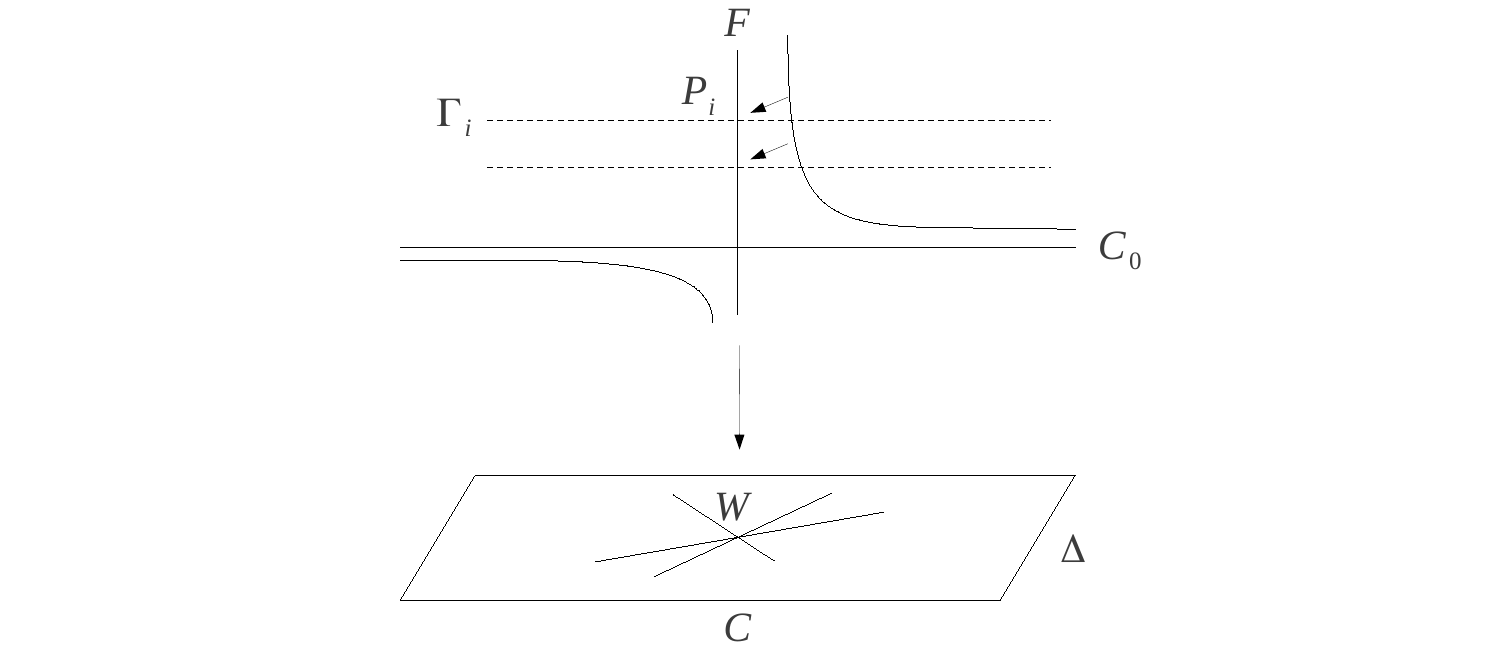}
\caption{The blowup $Z \rightarrow C \times \Delta$}
\label{eleven_fig}
\end{figure}

\subsection{Semistability} We take $\Gamma = \Gamma_1 + \dots + \Gamma_{11}$ where $\Gamma_i$ is the horizontal section of $\g{S}_0 = C\times{\bf P}^1$ through $P_i$ (see Fig. \ref{eleven_fig}). Consider the blowup $\wt{\g{S}} \rightarrow \g{S}$ at the relative points $\ms{P}_i$. We claim that for each $i=1,\dots,11$, $\ms{N}_{\wt \Gamma_i/\wt{\g{S}}}$ is indecomposable of degree 1 (hence semistable of slope 1/2). First, we will show that $\ms{N}_{\Gamma_i/\g{S}}$ is indecomposable of degree 0.

We will need some deformation theory. Let $D = {\mathbb{C}} [t]/t^2$ be the ring of dual numbers. Let $\g{S}' = \g{S} \times_\Delta D$ viewed as an infinitesimal deformation of $\g{S}_0 \cong C \times {\bf P}^1$ over $D$. We will say that a section $T$ of $\g{S}_0$ is {\it (infinitesimally) unobstructed} if and only if $T$ can be extended to a subscheme $T'$ of $\g{S}'$ flat over $D$.

\begin{lemma} Assume the above setting.
\begin{enumerate}
\item $T$ is unobstructed if and only if the following short exact sequence splits:
$$0 \rightarrow \ms{N}_{\g{S}_0/\g{S}}|_T \rightarrow \ms{N}_{T/\g{S}} \rightarrow \ms{N}_{T/\g{S}_0} \rightarrow 0$$ 
\item Suppose there are 3 disjoint horizontal sections $T_1,T_2,T_3$ of $\g{S}_0 \cong C \times {\bf P}^1$ that are unobstructed. Then, $\g{S}'$ is an infinitesimally trivial deformation, i.e. $\g{S}' \cong \g{S}_0 \times D$. In particular, any section of $\g{S}_0$ is unobstructed.
\item $C_0$ is obstructed.
\end{enumerate}
\end{lemma}
\begin{proof} 
(a) This is clear.

(b) We have $\g{S}' = {\bf P}(\ms{E}')$ where $\ms{E}' = \ms{E} \otimes D$. For any $i=1,2,3$, the embedding $f_i : T'_i \rightarrow \g{S}'$ induces a surjective morphism $\ms{E}' \rightarrow \ms{L}'_i$ where $\ms{L}'_i = f_i^* \ms{O}_{\g{S}'}(1)$ is a line bundle on $C' = C \times D$. By Nakayama's lemma, for any $i \neq j$, the induced map $\ms{E}' \rightarrow \ms{L}'_i \oplus \ms{L}'_j$ is an isomorphism. Hence $\ms{L}'_1 \cong \ms{L}'_2 \cong \ms{L}'_3 \cong \op{coker}(\ms{E}' \rightarrow \ms{L}'_1 \oplus \ms{L}'_2 \oplus \ms{L}'_3)$, i.e. $\g{S}' \cong \g{S}_0 \times D$.

(c) Consider the short exact sequence
$$0 \rightarrow \ms{N}_{Z/\g{S}}|_{C_0} \rightarrow \ms{N}_{C_0/\g{S}} \rightarrow \ms{N}_{C_0/Z} \rightarrow 0$$
where $\ms{N}_{Z/\g{S}}|_{C_0} \cong \ms{O}_C(-W)$ and $\ms{N}_{C_0/Z} \cong \ms{O}_C(W)$. Since $h^0(C,\ms{O}_C(W)) = 1$, it is clear that $\ms{N}_{C_0/\g{S}} \ncong \ms{O}_C \oplus \ms{O}_C$. Now consider the short exact sequence from part (a):
$$0 \rightarrow \ms{N}_{\g{S}_0/\g{S}}|_{C_0} \rightarrow \ms{N}_{C_0/\g{S}} \rightarrow \ms{N}_{C_0/\g{S}_0} \rightarrow 0$$ 
where $\ms{N}_{\g{S}_0/\g{S}}|_{C_0}\cong \ms{O}_C$ and $\ms{N}_{C_0/\g{S}_0} \cong \ms{O}_C$. It follows that $\ms{N}_{C_0/\g{S}_0}$ is indecomposable of degree 0. Hence, $C_0$ is obstructed.
\end{proof}

Part (c) of the lemma implies that $\g{S}'$ is not an infinitesimally trivial deformation. By part (b) and by symmetry, $\Gamma_i$ is obstructed for any $i$. It follows that the conormal bundle $\ms{N}_{\Gamma_i/\g{S}}$ is indecomposable of degree 0.

Denote by $P'_i$ the image of $\ms{P}_i$ in ${\bf P}(\ms{N}_{\Gamma_i/\g{S}})$. Clearly, $P'_i$ does not belong to the unique minimal section of ${\bf P}(\ms{N}_{\Gamma_i/\g{S}})$ (because $\ms{P}_i$ meets $\g{S}_0$ transversely). Finally, ${\bf P}(\ms{N}_{\wt \Gamma_i / \wt {\g{S}}})$ is obtained from ${\bf P}(\ms{N}_{\Gamma_i/\g{S}})$ by performing an elementary transform at $P'_i$. It follows that $\ms{N}_{\wt \Gamma_i / \wt {\g{S}}}$ is indecomposable of degree 1.

\subsection{Invariants}
Denote $\mu_i = \op{mult}_{\Gamma_i/\g{S}}(\ms{C})$. By symmetry, $\mu_1 = \dots = \mu_{11}$. Since $\Gamma_i \equiv C_0$, we have the following upper bound:
$$\mu_1 \leq \frac \mu { 11 } \leq \frac {d} {11 \cdot 3}. \eqno(\sharp)$$
The lower bound from Basic Lemma is:
$$\frac 1 2 \mu_1 \geq \gamma = 10m - 3d.\eqno(\flat)$$
Combining $(\sharp)$ and $(\flat)$, we get:
$$\frac 1 2 \cdot \frac d { 33 } \geq 10m - 3d \Longleftrightarrow 199d \geq 660m.$$
This completes the proof for eleven points.

\begin{remark}
It might be also profitable to study the behavior of $\ms{C}$ along the fiber $F$. We have $\ms{N}_{\wt F/\wt{\g{S}}} \cong \ms{O}_{{\bf P}^1}(10) \oplus \ms{O}_{{\bf P}^1}(1)$, which follows from the split exact sequence
$$0 \rightarrow \ms{N}_{\wt{Z}/\wt{\g{S}}}|_{\wt F} \rightarrow \ms{N}_{\wt F/\wt{\g{S}}} \rightarrow \ms{N}_{\wt F/\wt Z} \rightarrow 0$$
with $\ms{N}_{\wt{Z}/\wt{\g{S}}}|_{\wt F} \cong \ms{O}_{{\bf P}^1}(10)$ and $\ms{N}_{\wt F/\wt Z} \cong \ms{N}_{F/Z} \cong \ms{O}_{{\bf P}^1}(1)$. The fact that $\ms{N}_{\wt F/\wt{\g{S}}}$ is unstable causes certain multiplicity and tangency conditions on the limit curve $\ms{C}_0$ at the point $W=C_0 \cap F$. A more careful analysis of the situation is beyond of the scope of this paper.
\end{remark}

\bigskip

\section{Twelve Points}\label{sec_twelve}

In this section we prove Main Theorem for $n=12$. This case is similar to $n=10$.

\subsection{Setup} As before, $C$ is a smooth cubic in ${\bf P}^2$. Fix a point $W$ on $C$ such that $9W \sim 3H$. Take $A = \ms{O}_C(3W)$ which is of degree $\alpha = 3$. It is easy to see that for a general $\xi \in \op{Ext}^1(A,\ms{O}_C)$, $S(C,A,\xi)$ is an indecomposable elliptic ruled surface of degree 1. It follows that $C \equiv C_0 + f$ where $C_0$ is a minimal section of $S$.

\subsection{Degeneration}
Let $\g{S} = S \times \Delta$ and $Z = C \times \Delta \subset \g{S}$. We identify $S$ with the special fiber of $\g{S}$. Take $\Gamma = \Gamma_1 + \Gamma_2$ where each $\Gamma_i$ is a general section of the pencil $|-2K_S|$. In particular, $\Gamma_i \cdot C = (4C_0 - 2f) \cdot (C_0 + f) = 6$. Let $C \cap \Gamma_1 = \{P_1,\dots,P_6\}$ and $C \cap \Gamma_2 = \{P_7,\dots,P_{12}\}$. It follows that 
$$\sum P_i \sim 4 A \sim A + 3H.$$
Next, we specialize $\ms{P}_i$ in $Z$ to $P_i$ in a general way such that 
$$\sum \ms{P}_{i,t} \sim A + 3H$$
for any $t \in \Delta$.

\subsection{Semistability} We have to show that if the specialization of the points $\ms{P}_i$ is general enough, the conormal bundle $\ms{N}_{\wt \Gamma_i / \wt {\g{S}}}$ is semistable (of slope 3). Since semistability is an open property (\mbox{\cite{M2}, Thm. 2.8}), it suffices to describe a particular specialization for which $\ms{N}_{\wt \Gamma_i / \wt {\g{S}}}$ is semistable. This is not hard. In fact, we claim that we can specialize the points in such a way that
$$\ms{N}_{\wt \Gamma_1 / \wt {\g{S}}} \cong \ms{O}_{\Gamma_1}(P_1+P_2+P_3) \oplus \ms{O}_{\Gamma_1}(P_4+P_5+P_6)$$ 
and similarly for $\Gamma_2$. This can be achieved by moving the triples of points $\{\ms{P}_{3i-2},\ms{P}_{3i-1},\ms{P}_{3i}\}$ ``in parallel'' while being assigned to the same section of the pencil $|-2K_S|$.

\begin{figure}[ht]
\includegraphics[scale=.8]{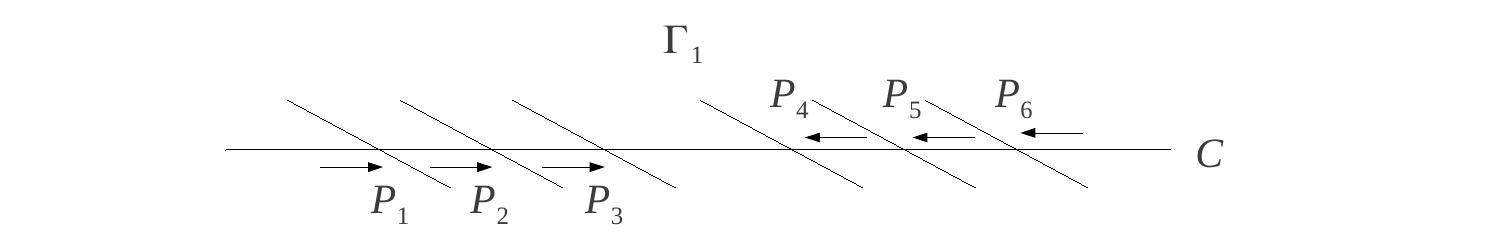}
\caption{Specialization to $\Gamma$ $(n=12)$}
\label{twelve_fig}
\end{figure}

\subsection{Invariants} Denote $\mu_i = \op{mult}_{\Gamma_i/\g{S}}(\ms{C})$. By symmetry, $\mu_1 = \mu_2$. Since $\Gamma_i \equiv 4 C_0 - 2f$, we have the following upper bound:
$$\mu_1 \leq \frac \mu { 2 \cdot 4 } \leq \frac {d} { 2 \cdot 4 \cdot 3 }. \eqno(\sharp)$$
The lower bound from Basic Lemma is:
$$\frac 1 2 \mu_1 \geq \gamma = 7m - 2d.\eqno(\flat)$$
Combining $(\sharp)$ and $(\flat)$, we get:
$$\frac 1 2 \cdot \frac d { 24 } \geq 7m - 2d \Longleftrightarrow 97d \geq 336m.$$
This completes the proof for twelve points.

\bigskip

\section{A Refinement}\label{sec_refinement}

Here we prove a certain refinement of the Main Theorem in the case of $n=10,11$ and $12$ points. We work in the setting of the previous sections. The idea is to show that, under some additional assumptions, the inequality $(\flat)$ can be replaced by a stronger inequality $(\flat\flat)$. First, we have:

\begin{lemma} Let $n=10,11$ or $12$. Consider the degeneration described above for the particular value of $n$. Let $\g{b} := \ms{B}_0 = \ms{O}_C(\sum m P_i - dH)$.

\begin{enumerate}
\item We have $\g{b} - bW \sim d(3W - H)$. In particular, $3\g{b} \sim 3bW$.
\item Assume there is an equality in $(\flat)$, i.e. $\frac 1 2 \mu_1 = \gamma$. Then $\g{b} \sim bW$.
\end{enumerate}
\end{lemma}
\begin{proof}
(a) We have $\sum P_i \sim n W$. Therefore
\begin{align*}\g{b} - bW &\sim (nmW - dH) - (nm - 3d)W  \\
&\sim d(3W-H).
\end{align*}
Since $9W \sim 3H$, it follows that $3\g{b} \sim 3bW$.

(b) Assume $\frac 1 2 \mu_1 = \gamma$. By Basic Lemma, there is an injective morphism:
$$\ms{O}_{\wt \Gamma_1}(-\wt{\ms{C}}) \hookrightarrow \op{Sym}^{2\gamma} \ms{N}_{\wt \Gamma_1/\wt{\g{S}}}.$$
Now the idea is to show that the vector bundle on the right hand side decomposes as a direct sum of line bundles of the same degree as $\ms{O}_{\wt \Gamma_1}(-\wt{\ms{C}})$. It will follow that $\ms{O}_{\wt \Gamma_1}(-\wt{\ms{C}})$ is isomorphic to one of the summands. Below we consider each case for $n$ separately.

\smallskip

{\it Case n=10.} Here $\ms{O}_{\wt \Gamma_1}(-\wt{\ms{C}})$ is of degree $2\gamma$. More precisely, since $\ms{O}_{\Gamma_1}(C) \cong \ms{O}_{\Gamma_1}(P_1+P_2)$, we have:
$$\ms{O}_{\wt \Gamma_1}(-\wt{\ms{C}}) \cong \ms{O}_{\Gamma_1}(-m(P_1 + P_2) + \g{b}f).$$
Recall that $\ms{N}_{\wt \Gamma_1/\wt{\g{S}}} \cong \ms{O}_{\Gamma_1}(P_1) \oplus \ms{O}_{\Gamma_1}(P_2)$. Therefore:
$$\op{Sym}^{2\gamma} \ms{N}_{\wt \Gamma_1/\wt{\g{S}}} \cong \bigoplus_{i+j = 2\gamma} \ms{O}_{\Gamma_1}(iP_1 + jP_2).$$
It follows that $\ms{O}_{\wt \Gamma_1}(-\wt{\ms{C}})$ must be isomorphic to one of the summands. Hence $\g{b}f \sim iP_1 + jP_2$ on $\Gamma_1$ for some $i,j$ with $i+j = 4b$. Since $\Gamma_1 \rightarrow C$ is an isogeny of degree 4, it follows that $4\g{b} \sim i P_1 + jP_2$ on $C$. By symmetry, $4\g{b} \sim jP_1 + iP_2$. Therefore, $8\g{b} \sim 4b(P_1+P_2) \sim 8bW$ on $C$. Since $3\g{b} \sim 3bW$ and $gcd(8,3)=1$, it follows that $\g{b} \sim bW$.

\smallskip

{\it Case n=11.} Here $\ms{O}_{\wt \Gamma_1}(-\wt{\ms{C}})$ is of degree $\gamma$. More precisely, since $\ms{O}_{\Gamma_1}(Z_0) \cong \ms{O}_{\Gamma_1}(P_1)$, we find:
$$\ms{O}_{\wt \Gamma_1}(-\wt{\ms{C}}) \cong \ms{O}_{\Gamma_1}(-mP_1 + \g{b}f) \cong \ms{O}_C(-mW + \g{b}).$$
Recall that $\ms{N}_{\wt \Gamma_1/\wt{\g{S}}}$ is indecomposable of degree 1 (determinant $W$). From the results in \mbox{Appendix \ref{appendix_A}}, $\op{Sym}^{2\nu}(\ms{N}_{\wt \Gamma_1/\wt{\g{S}}})$ is a direct sum of line bundles of the form $\ms{O}_C(\nu W + L_i)$ where $L_i^{\otimes 2} \cong \ms{O}_{C}$. We conclude that $\ms{O}_{\wt \Gamma_1}(-\wt{\ms{C}})$ is isomorphic to one of the summands. It follows that $2\g{b} \sim 2bW$. Since $3\g{b} \sim 3bW$ and $gcd(2,3)=1$, we conclude that $\g{b} \sim bW$.

\smallskip

{\it Case n=12.} This is similar to the case of ten points. We leave the details to the reader.

\end{proof}

The following result is a refinement of the Main Theorem. Note that it only applies when $3 \nmid d$.

\begin{proposition}\label{prop_refinement} Let $n=10,11$ or $12$. If $\ms{L}(d,n,m)$ is nonempty and $3\nmid d$, then $\kappa_n \geq 0$ with 
$$
\begin{array}{lcrcrcl}
\kappa_{10} &=& 721d &-& 2280 m &-& 60\\
\kappa_{11} &=& 199d &-& 660 m &-& 33\\
\kappa_{12} &=&  97d &-& 336m &-& 24.
\end{array}
$$
\end{proposition}

\begin{proof}
Fix $W$ so that $9W \sim 3H$ but $3W \nsim H$. Since $3 \nmid d$, part (a) of the lemma implies that $\g{b} \nsim bW$. From part (b), we get:
$$\frac 1 2 \mu_1 \geq \gamma + \frac 1 2. \eqno(\flat\flat)$$
Finally, $(\sharp)$ together with $(\flat\flat)$ imply the desired inequality.
\end{proof}

\begin{corollary} The following linear systems $\ms{L}(d,n,m)$ with \mbox{$v=-1$} are empty, hence non-special:
$$
\begin{array}{rrr|r|rrrr|r|rr}
d	& n	& m	& \chi_{{\bf P}^2}	& \mu	& \epsilon	& b	& \wh m	& \chi_{S}	& \gamma	& \kappa_{n}\\
\hline
1499	&10	&474	&0	&499	&2	&243	&25	&0	&12	&-1 \\
778	&10	&246	&0	&259	&1	&126	&13	&0	&6	&-2 \\
428	&11	&129	&0	&142	&2	&135	&13	&0	&6	&-1 \\
229	&11	&69	&0	&76	&1	&72	&7	&0	&3	&-2 \\ 
215	&12	&62	&0	&71	&2	&99	&9	&0	&4	&-1 \\
118	&12	&34	&0	&39	&1	&54	&5	&0	&2	&-2
\end{array}
$$
\end{corollary}

\begin{remark} The assumption $3 \nmid d$ in the proposition cannot be dropped. For example, consider the nonempty linear system $\ms{L}(57,10,18)$ (with $v=0$, $\kappa_{10} = -3$ and $3\mid d$). Unfortunately, it is not clear to us how to extend the proposition in the case $3 \mid d$. Our discussion will not be complete without mentioning the following interesting open problems: $\ms{L}(2220,10,702)$, $\ms{L}(627,11,189)$ and $\ms{L}(312,12,90)$ (with $v=0$, $\kappa_n = 0$ and $3\mid d$).
\end{remark}

\bigskip

\section{The Remaining Case}\label{sec_last}

Here we prove the Main Theorem in the case when $k \geq 3$, $\alpha$ is even, $\alpha \mid 2n$. This generalizes the case of eleven points (in fact, the proof can be also applied in the case of eight points).

\subsection{Setup} Let $C$ be a smooth plane curve of degree $k$. We will make the following assumption: there is a divisor $W = W_1 + \dots + W_{\frac \alpha 2}$ on $C$, where $W_i$'s are distinct points, such that
$$\tfrac {2k^2}\alpha W \sim kH.$$
Here is one way to construct such a curve. Fix a line $\ell \subset {\bf P}^2$ and let $W_1,\dots,W_{\frac \alpha 2}$ be distinct points on $\ell$. Now, take $C$ to be any smooth curve of degree $k$ which is tangent to $\ell$ to order $\frac{2k}\alpha$ at each of the points $W_i$. It follows that $\frac{2k}\alpha W \sim H$, which satisfies the assumption.

\subsection{First Degeneration} It will be convenient to re-index the $n$ relative points as $\{\ms{P}_{ij}\}$ where $i=1,\dots,\frac {2n}\alpha$ and $j=1,\dots,\frac \alpha 2$. We will construct a relative marked ruled surface $\g{S}(Z,\ms{A},\xi;\{\ms{P}_{ij}\})$ such that:

\begin{itemize}
\item The special fiber $\g{S}_0$ is simply $C \times {\bf P}^1$.
\item The special section $Z_0 = C_0 \cup F_1 \cup \dots \cup F_{\frac \alpha 2}$; here $C_0$ is a horizontal section of $\g{S}_0$ and $F_j$ is the fiber of $\g{S}_0$ above $W_j$, for each $j$.
\item The relative points $\{\ms{P}_{ij}\}$ on $Z$ are such that $\sum \ms{P}_{ij,t} \sim \ms{A}_t + kH$ on $Z_t \cong C$, for general $t$.
\item Each limit point $P_{ij} = \ms{P}_{ij,0}$ is a general point on the fiber $F_j$.
\end{itemize}

The construction generalizes the case of eleven points. Namely, we first choose relative points $\overline{\ms{P}}_{ij}$ in $C \times \Delta$ specializing to $W_j \times \{0\}$ in a general way. Next, let $\ms{A}' = \ms{O}_{C\times\Delta}(\sum \ms{P}_{ij} - kH)$ and $\ms{A} = \ms{A}' \otimes I_{W\times\{0\}}$. It follows that
\begin{align*}
\ms{A}'_0 &\cong \ms{O}_{C}(\tfrac {2n}\alpha W - kH) \cong \ms{O}_C(2W);\\
\ms{A}_0 &\cong \ms{O}_C(W) \oplus \ms{O}_W(W).
\end{align*}
Consider the short exact sequence on $C \times \{0\}$:
$$0 \rightarrow \ms{O}_{C} \rightarrow \ms{O}_C(W) \oplus \ms{O}_C(W) \rightarrow \ms{O}_C(W) \oplus \ms{O}_W(W) \rightarrow 0$$
(Note that $h^0(C,\ms{O}_C(W)) = 1$, so the sequence is unique). By Prop. \ref{prop_smooth}, the sequence can be extended to
$$0 \rightarrow \ms{O}_{C\times \Delta} \rightarrow \ms{E} \rightarrow \ms{A} \rightarrow 0$$
over $C \times \Delta$, where $\ms{E}$ is locally free. Finally, we take $\g{S} = {\bf P}(\ms{E})$ and $Z = {\bf P}(\ms{A})$. It follows that $\g{S}_0 \cong {\bf P}(\ms{O}_C \oplus \ms{O}_C) \cong C \times {\bf P}^1$. We take $\ms{P}_{ij}$ to be the strict transform of the $\overline{\ms{P}}_{ij}$ on the blowup $Z \rightarrow C \times \Delta$ at $W \times \{0\}$.

Next, we will distinguish between two cases: $\alpha=2$ and $\alpha \geq 4$.

\subsection{Semistability $(\alpha=2)$} Let $\Gamma_i$ be the horizontal section of $\g{S}_0$ through $P_{i,1}$. Just as in the case of eleven points, we can show that the conormal bundle $\ms{N}_{\wt \Gamma_i/\wt{\g{S}}}$ is semistable of slope 1/2.

\subsection{Second Degeneration $(\alpha \geq 4)$} In this case, we perform another degeneration on the trivial ruled surface $S = C \times {\bf P}^1$.
Fix $\frac {2n}\alpha$ general horizontal sections $\Gamma_1,\dots,\Gamma_{\frac{2n}\alpha}$ of $S$. Let $P_{ij} = \Gamma_i \cap F_j$ for $i=1,\dots,\frac{2n} \alpha$ and $j=1,\dots,\frac \alpha 2$. Next, we specialize the $n$ relative points $\ms{P}_{ij}$ to $P_{ij}$ by ``sliding'' them along the corresponding fibers $F_j$, in a general way.

\subsection{Semistability $(\alpha \geq 4)$} Denote by $\wt{S\times\Delta}$ the blowup of $S \times \Delta$ at the relative points $\ms{P}_{ij}$ constructed in the previous step. Now, ${\bf P}(\ms{N}_{\wt \Gamma_i/ \wt{S \times \Delta}})$ is obtained from ${\bf P}(\ms{N}_{\Gamma_i/S \times \Delta}) = {\bf P}(\ms{O}_C \oplus \ms{O}_C)$ by applying $\alpha/2$ elementary transforms at general points on the fixed fibers through $W_1,\dots,W_{\frac \alpha 2}$. Since $\alpha/2 \geq 2$, it follows that the resulting vector bundle is semistable of slope $\alpha/4$ (the proof is similar to that of Lemma \ref{lemma_ss_fiber}).

\subsection{Invariants} The computation of invariants was carried out in Section \ref{sec_main}. This completes the proof of the Main Theorem.

\bigskip

\appendix
\section{The Indecomposable Elliptic Ruled Surface of Degree 1}\label{appendix_A}

Below we summarize some facts about the indecomposable elliptic ruled surface of degree 1. Our references are \cite{A} and (\cite{H}, Chapter V.2).

Let $C$ be an elliptic curve. Let $\ms{E}$ be an indecomposable rank 2 vector bundle of degree 1 on $C$. Then, $\ms{E}$ arises as the unique nontrivial extension
\begin{align}\label{ext1}
0 \rightarrow \ms{O}_C \rightarrow \ms{E} \rightarrow A \rightarrow 0
\end{align}
where $A = \op{det}(\ms{E})$.

Let us compute the symmetric powers of $\ms{E}$. By (\cite{A}, Lemma 22 on p.439), we have: $$\ms{E}^* \otimes \ms{E} \cong \ms{O}_C \oplus L_1 \oplus L_2 \oplus L_3$$ where the $L_i$ are the nontrivial line bundles with $L_i^{\otimes 2} \cong \ms{O}_C$. Also, by (\cite{A}, Cor. to Thm. 7 on p.434), we have $\ms{E}\otimes L_i \cong \ms{E}$ and $\ms{E}^* \cong \ms{E} \otimes A^{-1}$. Finally, we have the Clebsch-Gordan formula (\cite{A}, p.438) for a rank 2 vector bundle:
$$\op{Sym}^m \ms{E} \otimes \ms{E} \cong \op{Sym}^{m+1} \ms{E} \oplus (A \otimes \op{Sym}^{m-1} \ms{E}).$$
Using the above, we find:
\begin{align*}
\op{Sym}^2 \ms{E} &\cong A \otimes (L_1 \oplus L_2 \oplus L_3) \\
\op{Sym}^3 \ms{E} &\cong A \otimes (\ms{E} \oplus \ms{E}) \\
\op{Sym}^4 \ms{E} &\cong A \otimes (\ms{O}_C^{\oplus 2} \oplus L_1 \oplus L_2 \oplus L_3) \\
\dots
\end{align*}

In general, if $m=2k$ is even, $\op{Sym}^m \ms{E}$ decomposes as a sum of line bundles that are isomorphic to $A^{\otimes k}$ or $A^{\otimes k} \otimes L_i$. If $m=2k+1$ is odd, $\op{Sym}^m \ms{E}$ decomposes as a sum of $k+1$ copies of $A^{\otimes k} \otimes \ms{E}$.

Next, consider the ruled surface $S = {\bf P}(\ms{E})$ together with the projection $\pi:S \rightarrow C$. We identify $C$ with the unique section  of $|\ms{O}_{S}(1)|$. The anticanonical class of $S$ is $$-K_S \sim 2C - \pi^*(A).$$

Note that $$K_S^2 = 0.$$

We have:

\begin{proposition}\label{prop_elliptic} Let $\ms{E}$ be the indecomposable vector bundle of rank 2 and degree 1, $\op{det}(\ms{E})=A$. Consider the ruled surface $S = {\bf P}(\ms{E})$.
\begin{enumerate}
\item For any $x \in \op{Pic}^0(C)$, there is a unique curve $C_x \sim C + \pi^*(x)$ on $S$.
\item There are precisely 3 curves $\Gamma_i$, $i=1,2,3$, on $S$ that are numerically equivalent to $-K_S$. Their rational equivalence classes are given by $\Gamma_i \sim 2C - \pi^*(A + L_i)$ for each nontrivial line bundle $L_i$ with $L_i^{\otimes 2} \cong \ms{O}_C$.
\item The linear system $|-2K_S|$ sweeps a base-point free pencil on $S$. There are 3 nonreduced sections, namely $2\Gamma_i$, $i=1,2,3$. Any other section is a smooth elliptic curve $\Gamma$ isomorphic to $C$ (the natural projection $\Gamma \rightarrow C$ being the usual multiplication-by-2 map). 
\end{enumerate}
\end{proposition}

\begin{proof} 

a) This follows from the fact that $\ms{E} \otimes \pi^{*}(x)$ is the unique indecomposable rank 2 vector bundle with determinant $A + 2x$.

b) This follows from $\op{Sym}^2 \ms{E} \cong A \otimes (L_1 \oplus L_2 \oplus L_3).$

c) We have $\op{Sym}^4 \ms{E} \cong A^{\otimes 2} \otimes (\ms{O}_C^{\oplus 2} \oplus L_1 \oplus L_2 \oplus L_3).$ Therefore $h^0(-2K_S) = h^0(\op{Sym}^4 \ms{E} \otimes A^{-2}) = 2$, i.e. $|-2K_S|$ sweeps a pencil on $S$. Next, let $\Gamma$ any section of $|-2K_S|$ other than $2\Gamma_i$, $i=1,2,3$. From part b), $\Gamma$ is irreducible. Since $\Gamma$ is of arithmetic genus 1, it follows that $\Gamma$ is a smooth elliptic curve. One can show that $\Gamma$ is isomorphic to $C$ as follows. First, one shows that every irreducible section of $|-2K_S|$ is isomorphic to the fixed section $\Gamma$. Next, one checks that, for any $i=1,2,3$, $\Gamma$ admits a 2:1 cover to $\Gamma_i$. It follows that $\Gamma\rightarrow \Gamma_i$ is the isogeny dual to $\Gamma_i \rightarrow C$.
\end{proof}

\bigskip

\section{Some Continued Fractions}\label{appendix_B}
Let $n = k^2 + \alpha$ where $k>0$ and $\alpha>0$. Consider the matrix
$$M_1 =  \left[
\begin{array}{cc}
k & n \\
1 & k \\
\end{array}
\right].
$$
For any positive integer $i$, define
$$M_i = \left[
\begin{array}{cc}
p_i & q_i \\
r_i & p_i \\
\end{array}
\right] = \alpha^{-\lfloor i/2 \rfloor}(M_1)^i.$$
Note that $\det (M_{2i-1}) = -\alpha$ and $\det (M_{2i}) = 1$.

For example,
\begin{align*}
p_1 = k;\ \ \ q_1 = n;\ \ \ r_1 = 1;
\end{align*}
and
\begin{align*}
p_2 = \frac{n+k^2}\alpha;\ \ \ q_2 = \frac{2nk} \alpha;\ \ \ r_2 = \frac{2k}\alpha.
\end{align*}

The ratios $p_i/r_i$ and $q_i/p_i$ have natural expansions as continued fractions approximating $\sqrt{n}$ (the later are palindromic). In particular,
$$\frac{q_2}{p_2} = c^{(1)}_n \ \ \ \ \mbox{and} \ \ \ \ \ \frac{q_4}{p_4} = c^{(2)}_n$$
are precisely the constants in Theorem \ref{thm1} and the Main Theorem.

\bigskip

\bigskip

\begin{lemma}\label{lemma_B} Let $d$ and $m$ be any real numbers.

\begin{enumerate}
\item The following are equivalent:
$$\frac \alpha 2 \frac d {p_1} \geq q_1 m - p_1 d\ \ \ \Longleftrightarrow\ \ \  p_2 d - q_2 m \geq 0.$$

\item The following are equivalent:
$$\frac 1 2 \frac d { q_2 }  \geq p_2 m - r_2 d\ \ \ \Longleftrightarrow\ \ \ p_4 d - q_4 m \geq 0.$$
\end{enumerate}

\end{lemma}

\begin{proof}
(a) Multiply both sides by $2p_1$ and substitute $\alpha = -p_1^2 + q_1 r_1$:
$$(\alpha + 2p_1^2) d \geq 2p_1q_1 m \Longrightarrow \underbrace{(p_1^2 + q_1r_1)}_{\alpha p_2} d \geq \underbrace{2p_1q_1}_{\alpha q_2} m.$$

(b) Multiply both sides by $2q_2$ and substitute $1 = p_2^2 - q_2 r_2$:
$$(1 + 2q_2 r_2) d \geq 2p_2 q_2 m \Longrightarrow \underbrace{(p_2^2 + q_2 r_2)}_{p_4} d \geq \underbrace{2p_2 q_2}_{q_4} m.$$

\end{proof}

\bigskip

\bibliographystyle{amsplain}

\end{document}